\documentclass[a4paper,11pt]{article}

\usepackage[scale=0.76]{geometry}

\usepackage{libertine}

\usepackage{amsmath}
\usepackage{amssymb}
\usepackage{mathtools}
\usepackage{mathrsfs}
\usepackage{bbm}
\usepackage{amsthm}

\usepackage[parfill]{parskip}
\usepackage{graphicx}
\usepackage{epstopdf}
\usepackage{caption}
\usepackage{enumitem}
\usepackage{etoolbox}
\usepackage{calc}

\usepackage[maxnames=6, backend=bibtex, url=false, isbn=false, doi=false, style=alphabetic]{biblatex}
\renewbibmacro{in:}{}
\usepackage[affil-sl]{authblk}
\usepackage{hyperref}
\usepackage{orcidlink}

\newcommand{\pasymp}{%
  \mathrel{\overset{
    \mathmakebox[\widthof{$\scriptstyle\mathrm{whp}$}][c]{\mathbb P}
  }{\asymp}}%
}

\newcommand{\whpasymp}{%
  \mathrel{\overset{\mathrm{whp}}{\asymp}}%
}

\newtheorem{theorem}{Theorem}[section]
\newtheorem{corollary}[theorem]{Corollary}
\newtheorem{proposition}[theorem]{Proposition}
\newtheorem{lemma}[theorem]{Lemma}

\theoremstyle{definition}
\newtheorem{definition}[theorem]{Definition}
\AtBeginEnvironment{definition}{%
  \pushQED{\qed}%
}
\AtEndEnvironment{definition}{\popQED}

\newcommand{\cC}{\mathcal{C}}
\newcommand{\cE}{\mathcal{E}}
\newcommand{\cG}{\mathcal{G}}
\newcommand{\cL}{\mathcal{L}}
\newcommand{\cR}{\mathcal{R}}

\newcommand{\de}{\operatorname{d}}

\DeclareMathOperator{\e}{\mathbb{E}}
\DeclareMathOperator{\p}{\mathbb{P}}
\DeclareMathOperator{\var}{\mathbb{V}{\rm ar}}

\title{The voter model on the hyperbolic graph}
\author[1]{John Fernley \orcidlink{0000-0002-6635-4341}}
\author[2]{Christian Hirsch}
\affil[1]{CRiSM, Department of Statistics, University of Warwick,
United Kingdom\\\href{mailto:john.fernley@warwick.ac.uk}{\rm john.fernley@warwick.ac.uk}}
\affil[2]{Department of Mathematics, Aarhus University, Denmark\\\href{mailto:christian.hirsch@math.au.dk}{\rm christian.hirsch@math.au.dk}}
\date{\today}
  
\begin{document}
\maketitle

\begin{abstract}
We consider the voter model on the giant component of a hyperbolic random graph, which is a spatial scale-free network, in the sparse and linear-giant regime $\alpha\in(1/2,1)$. We find that the quenched expected consensus time has order $n^{2-1/\alpha}$, as the number of vertices $n\to\infty$, with probability arbitrarily close to one. This is generalised to the voter model where each vertex changes its opinion at rates $q(v)=\de(v)^\varphi$, where we also establish the consensus time orders for all $\varphi\geq 0$. These orders have 3 regimes, with a phase transition at $\varphi=2-2\alpha$. For the upper bounds, our main proof idea is to connect the meeting set to some fixed target vertex of appropriate height in the product chain electrical network, to make rigorous an argument of  \cite{durrett2010some}.
\end{abstract}

\section{Introduction and results}

The voter model is one of the basic interacting particle systems.  It originates in the model for spatial conflict of Clifford and Sudbury
\cite{clifford1973model} and the work of Holley and Liggett
\cite{holley1975ergodic}; see also \cite{liggett85}.  Each site carries an
opinion and, at random times, adopts the opinion of another site, usually a neighbour in a graph environment.  On a finite connected graph the process eventually reaches a constant configuration, and a fundamental question is how the geometry of the graph determines the time to consensus.

\begin{definition}[Voter model]
Let $V$ be finite and let $Q=(Q_{ij})_{i,j\in V}$ be an irreducible Markov generator, with $q(i):=-Q_{ii}$.  For $\xi\in\{0,1\}^V$ and $i,j\in V$, let $\xi^{i\leftarrow j}$ be obtained from $\xi$ by replacing the opinion at $i$ by that at $j$.  The voter model $(\xi_t)_{t\geq0}$ with copying kernel $Q$ is the Markov process on $\{0,1\}^V$ with generator
\begin{equation}
\cL f(\xi) = \sum_{i\neq j}Q_{ij} \left(f(\xi^{i\leftarrow j})-f(\xi)\right).             \label{eq:voter-generator}
\end{equation}
Thus site $i$ updates at rate $q(i)$ and, when it updates, copies $j$ with probability $Q_{ij}/q(i)$.
\end{definition}

In the graphical construction, an arrow from $j$ to $i$ is drawn at rate $Q_{ij}$ and the opinion at $i$ follows the arrow.  Tracing arrows backwards gives a Markov chain with generator $Q$; several ancestral lineages move as copies of this chain and coalesce when they meet.  Write $T_{\rm coal}$ for the coalescence time when initially there is one lineage at every site, and
\[
T_{\rm cons}:=\inf\{t\geq0:\xi_t\text{ is constant on }V\}.
\]
The graphical coupling gives $T_{\rm cons}\leq T_{\rm coal}$ for binary opinions $\xi_0\in\{0,1\}^V$.  In the voter model begun with a distinct opinion at every site, the consensus time exactly equals $T_{\rm coal}$.  If $\pi$ is stationary, so that $\pi Q=0$, then the proportion of opinion $1$
\[
p_1(\xi):=\sum_{i\in V}\pi(i)\xi(i)
\]
defines a martingale $(p_1(\xi_t))_{t\geq0}$, for more details see \cite[Section~3]{coxchenchoi}. Now, let $X,Y$ be independent copies of the dual chain, and define the standard meeting times
\[
\begin{split}
T_{\rm meet}&:=\inf\{t\geq0:X_t=Y_t\},\\
t_{\rm meet}^\pi&:= \e_{\pi\otimes\pi}\left( T_{\rm meet}
\right),\\
t_{\rm meet}&:= \max_{x,y} \e_{x,y}\left( T_{\rm meet}
\right).\\
\end{split}
\]

If we consider this a hitting time for the two-dimensional Markov chain $Z=(X,Y)$, the exit distribution of $Z$ from the meeting set $\Delta=\{(v,v):v\in V\}$ is the probability measure on $i\neq j$ given by
\[
\rho(i,j)\propto \pi(i)^2 Q_{ij}+ \pi(j)^2 Q_{ji} .
\]

\begin{definition}[Initial condition]
On a fixed graph $\cG$ and with $u \in (0,1)$, $\mu_u:= \bigotimes_{v\in V(\cG)}{\rm Ber}(u)$ is the product measure.
\end{definition}

In the following result we give high probability bounds (see Section \ref{sec_notation} for notation) for the expected coalescence time on the configuration model graph (a uniform random graph from the subset of graphs with a given fixed degree sequence).

\begin{theorem}[Folklore]
The voter model on a configuration model $\cG$ defined by degrees $\{d_1,\dots,d_n\}$ with $\min_i d_i\geq 3$, $\max_i d_i = n^{1-\Omega(1)}$ and $\sum_{i=1}^n d(i)=O(n)$ satisfies
\[
\e_{\mu_u}(T_{\rm cons}|\cG) \whpasymp \frac{n^2}{\sum_{i=1}^n \de(i)^{2}}.
\]
\end{theorem}

\begin{proof}
The upper bound is precisely \cite[Theorem 7.4.1]{durrett2024dynamics} with a factor of the relaxation time $t_{\rm rel}$, for which we refer to \cite[Lemma 3.2]{gkantsidis2003conductance} who find that this graph is with high probability an expander, i.e. $t_{\rm rel}=O(1)$. The matching lower bound is  \cite[Equation 3.21]{coxchenchoi}.
\end{proof}

Note that the order of consensus on the configuration model is $(\sum\pi^2q)^{-1}$, as here for the simple random walk $\pi\propto d$ and $q\equiv1$. It might be thought that the consensus order agrees with the mean-field order because the underlying graph is an expander, but in this article we will see the same expression for a $1$-dimensional spatial graph with polynomially small spectral gap: there are different techniques for the three cases, but ultimately Theorem \ref{thm:consensus} is also the order of $(\sum\pi^2q)^{-1}$.

We use the following standard (Poissonised) version of the hyperbolic random graph with a fixed degree scaling parameter $\nu>0$, as opposed to the alternative where we condition on the Poisson process giving a fixed deterministic number of vertices. This parameter $\nu$ doesn't increase the vertex total, which is fixed at ${\rm Pois}(n)$, but the expected total number of edges grows linearly with $\nu$.

\begin{figure}
\includegraphics[width=\textwidth, keepaspectratio]{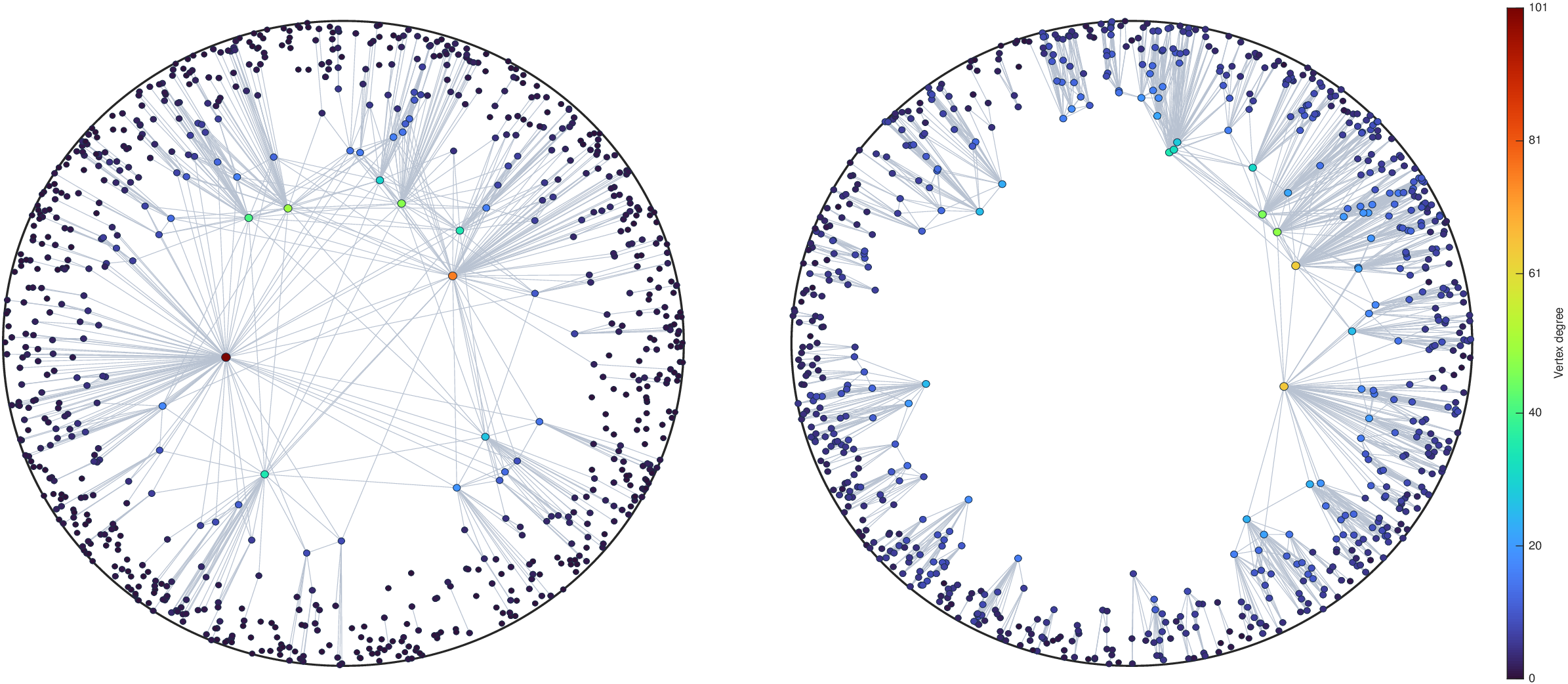} \caption{Two realisations of the hyperbolic random graph, both with $n=\e(|V|)=1000$ and $\nu$ such that $\e(|E|)=10,000$, but on the left we see $\alpha=9/16$ and on the right $\alpha=15/16$. Smaller $\alpha$, that is heavier tail degree distributions, have relatively more long--range edges. Colours denote vertex degree.}\label{fig_hyperbolic}
\end{figure}

\begin{definition}[Hyperbolic random graph]
Fix $\alpha\in(1/2,1)$ and $\nu>0$, and put $R:=2\log(n/\nu)$.  Let $\mathcal P_n$ be a Poisson point process on the hyperbolic disk $B_O(R)$ with intensity measure
\[
\lambda_n( \de r, \de\theta ) := \frac{n}{2\pi} \frac{\alpha\sinh(\alpha r)}{\cosh(\alpha R)-1} \mathbbm{1}_{\{0\leq r<R\}}\,\de r \de \theta, \qquad 0\leq\theta<2\pi.
\]

The hyperbolic random graph $G_{\alpha,\nu}(n)$ has vertex set $\mathcal P_n$, and two points are adjacent if and only if their hyperbolic distance is less than $R$.  Thus
\[
N_n:=|V(G_{\alpha,\nu}(n))|\sim{\rm Pois}(n).
\]

Throughout, $n$ denotes the mean vertex count rather than the realised value $N_n$.  This is the model of \cite[Section~2.3]{MR4816414}, with the density parameter encoded through $R$.
\end{definition}

By \cite[Theorem 5]{gugelmann2012random}, its largest degree is $n^{1/(2\alpha)+o(1)}$ with high probability.  Throughout, $\cG$ denotes its giant component.  The parameter region $\alpha\in(1/2,1)$ is the sparse scale--free regime in which a giant exists for every fixed $\nu>0$.

\begin{definition}%
The degree--dependent voter model is the voter model dual to the Markov chain
\begin{equation}
q(v)=\de(v)^\varphi, \qquad Q_{v,w}=\de(v)^{\varphi-1}\mathbbm{1}_{\{v\sim w\}},
                           \label{eq:voter-kernel}
\end{equation}
i.e. each vertex $v$ imitates each $w$ at rate $Q_{v,w}$, and so by iteration the opinion history has these dynamics after time reversal.
\end{definition}

For this parametrisation in \eqref{eq:voter-kernel}, introduced in
\cite{fernleyortgiese}, the stationary and post-meeting distributions
are
\[
\pi(i)\propto\de(i)^{1-\varphi}, \qquad \rho(i,j)\propto\mathbbm{1}_{i\sim j}\left( \de(i)^{1-\varphi} + \de(j)^{1-\varphi} \right).
\]

We can now state the main theorem: we introduce some notation for two meanings sometimes ascribed to $\Theta_{\p}$, bounds holding with high probability or \emph{in probability}, i.e. with probability close to $1$. Formal definitions are outlined in Section \ref{sec_notation}.

\begin{theorem}\label{thm:consensus}
Let $\alpha\in(1/2,1)$, $\nu>0$ and $\varphi\geq0$ be fixed.  On the giant component $\cG$ of $G_{\alpha,\nu}(n)$, start the degree--dependent voter model from $\mu_u$, where $u\in(0,1)$ is fixed. We find
\begin{align}
\text{if }\,0\leq\varphi<2-2\alpha \text{, then}\quad \e_{\mu_u}(T_{\rm cons}\mid\cG) &\pasymp n^{\,2-1/\alpha+\varphi/(2\alpha)} ,
  \label{eq_theorem_1}
\\[0.5em] \text{if }\,\varphi=2-2\alpha \text{, then}\quad \e_{\mu_u}(T_{\rm cons}\mid\cG) &\whpasymp \frac{n}{\log n} ,
  \label{eq_theorem_2}
\\
\text{and if }\,\varphi>2-2\alpha \text{, then}\quad \e_{\mu_u}(T_{\rm cons}\mid\cG) &\whpasymp n .
  \label{eq_theorem_3}
\end{align}
\end{theorem}

In the first regime,
\[
2-\frac{1}{\alpha}+\frac{\varphi}{2\alpha}<1
\]
and so the large degrees can speed up consensus to a sublinear order, but the influence of these large degrees is decreased as $\varphi$ increases. Increasing $\varphi$ cannot, however, make the quenched mean larger than order $n$.

The result could also be stated for the absorption time of the voter model from the Bernoulli measure $\mu_u$ on the full hyperbolic graph, using the result of \cite{MR4032854} that, with high probability, all other components are uniformly $O((\log n)^{1/(1-\alpha)})$.  This disconnected-graph result would then concern componentwise consensus rather than global consensus but the order would be the same.

In terms of the usual network parameter $\tau\in(2,3)$ for the tail decay of the empirical degree distribution, the orders in the theorem would translate to
\[
T_{\rm cons}\pasymp
\begin{cases}
n^{\,2-(2-\varphi)/(\tau-1)},
    & 2<\tau<3-\varphi,\\
n/\log n,
    & \tau=3-\varphi,\\
n, & 3-\varphi<\tau<3,
\end{cases}
\]
with all $3$ of these cases only being simultaneously possible when $\varphi\in(0,1)$. Note that removing the expectation and giving all three cases as orders in probability is an alternative presentation for Theorem \ref{thm:consensus}, but we find the expectation more informative.

The general message of this paper is perhaps that the heterogeneous mean-field heuristic of \cite{sood2008voter} continues to give the correct consensus order for the voter model on a scale-free network, even in a particular type of scale-free network model which is both slow-mixing and not locally treelike (see Figure \ref{fig_hyperbolic}). However, another slow-mixing instance of the voter model was very recently considered by \cite{koval2026meeting}, in that work $\varphi=0$ and they see in the case of a critical Erd\H{o}s--R\'enyi graph that a largest component of size $n^{2/3}$ with mixing time $n$ (see \cite{MR2435849}) in fact has consensus time of order $n$.

If we write $T_{\rm mix}$ for an optimal strong stationary time for the product chain, as provided by
\cite[Theorem~1]{fill1991time}, in general we have the easy coupling lower bound
\[
T_{\rm coal} \geq \max_v T_{\rm meet}^{\delta_v\otimes \pi} \geq T_{\rm mix}
\]
so the result of \cite{koval2026meeting} reveals that when the mixing time of the dual random walk is sufficiently slow it forces slow coalescence. This is related to the critical Erd\H{o}s--R\'enyi graph being globally tree-like, quite unlike the hyperbolic network model: for the hyperbolic model we still have $t_{\rm mix}\sum q \pi^2 \to 0$ and so the mean-field order is possible. Note however  \cite{koval2026meeting}  also demonstrate in the slightly-supercritical Erd\H{o}s--R\'enyi graph an example with $t_{\rm mix}\sum q \pi^2 \to 0$ for which the consensus order is not mean-field, so there is no straightforward general heuristic.

In this article, the proof relies centrally on Kac's formula \cite[Equation 2.24]{aldous-fill-2014} to get the stationary return time to the meeting set $\Delta=\{(v,v):v\in V\}$ of two independent walkers, which is in general a lower bound on the order of meeting time from stationarity. For $0\leq\varphi\leq2-2\alpha$, we consider the electrical network associated to the Markov chain of the two walkers, \emph{shorting} the diagonal set $\Delta$ by connecting each pair of states with a zero resistance edge. By putting a finite energy flow on this network from $\Delta$ to a target of appropriate height, we show that the random walkers can hit that target before they return to the meeting set. Then, by separately arguing that the target is hit after the mixing time, we show that stationary meeting is on the same order as the return time. This idea was previously used in \cite{MR4933840} on the dynamic graph but it is much easier to construct bounded energy flows on static graphs, particularly on hyperbolic graphs because we can adapt the construction of \cite{MR4816414} to the product Markov chain.

Instead for $\varphi>2-2\alpha$, the matching upper bound follows from the random-target time, or Kemeny's constant, for which we have a bound directly from \cite{MR4816414}. Then we can use Aldous's comparison, \cite[Proposition~2]{aldous91}, to bound the worst-case meeting time. Finally, the comparison of \cite{kanade23} passes from the meeting-time upper bounds to full coalescence, and hence to voter consensus.

\subsection{Asymptotic notation}\label{sec_notation}

All asymptotics are as $n\to\infty$.  Unless stated otherwise, ``with high probability'' refers to the randomness of the random graph under consideration and means with probability tending to one.  Implicit constants may depend on parameters declared fixed, but not on $n$.

For deterministic nonnegative sequences, the Landau notation $O$, $o$ and $\Omega$ have their usual meanings, and the relation $\sim$ is asymptotic equivalence i.e. ratio $1$.  We sometimes write $a_n\ll b_n$ and $a_n\gg b_n$ which are equivalent to $a_n=o(b_n)$ and $b_n=o(a_n)$, respectively. A tilde permits a fixed polylogarithmic loss: $\widetilde O(a_n)$ means $O(a_n(\log n)^C)$, and $\widetilde\Omega(a_n)$ means $\Omega(a_n/(\log n)^C)$, for some fixed $C$.

For a random sequence $X_n$ and a positive sequence $a_n$, which may itself be random, the statement $X_n=O(a_n)$ with high probability means that $\p(|X_n|\leq Ca_n)\to1$ for some fixed $C$; $X_n=o(a_n)$ with high probability means that this holds with $C$ replaced by some deterministic sequence tending to zero.  For nonnegative $X_n$, $X_n=\Omega(a_n)$ with high probability is defined analogously with a fixed positive lower constant.

Finally, we resolve some ambiguity with the notation $\Theta_{\p}$ by avoiding that notation and using the following for its two possible meanings. Given two nonnegative sequences $(X_n)_n$ and $(Y_n)_n$, for any $\epsilon>0$ there are constants $0<c<C<\infty$ such that for large $n$ we have:
\begin{table}[h!]
\centering \caption{The three types of $\asymp$ and their definitions.}
\begin{tabular}{c|l}
$X_n\asymp Y_n$ & $cY_n\leq X_n\leq CY_n$  \\[0.3em] \hline $X_n\whpasymp Y_n$ & $ \p(cY_n\leq X_n\leq CY_n)\longrightarrow1$ \\[0.3em] \hline $X_n\pasymp Y_n$ & $ \p(c Y_n\leq X_n\leq C Y_n) \geq1-\varepsilon$
\end{tabular}
\end{table}

\subsection{The voter model}

The following positive-escape criterion is the form of Aldous's ``Poisson clumping heuristic'' which is the centre of the proof in this article. This heuristic was applied in \cite[Equation 15]{durrett2010some} to the voter model with the assumption of a positive escape probability (condition~\ref{cond:clumping-escape} below). We will later be able to verify this condition for our model.

\begin{proposition}[Poisson clumping]\label{prop:poisson-clumping}
For a sequence of dual chains, suppose that for deterministic $t=t_n$,
\begin{enumerate}[label=\normalfont(\roman*)]
  \item\label{cond:clumping-collision} $\sum_v\pi(v)^2\to0$;
  \item\label{cond:clumping-window}
$t_{\rm mix}\ll t\ll(\sum_vq(v)\pi(v)^2)^{-1}$;
  \item\label{cond:clumping-escape}
$\p_\rho(T_{\rm mix}<t<T_{\rm meet})=\Omega(1)$.
\end{enumerate}
Then
\[
t_{\rm meet}^\pi\asymp \frac{1}{\sum_vq(v)\pi(v)^2}.
\]
\end{proposition}

\begin{proof}
For this proof put
\[
A:=\sum_vq(v)\pi(v)^2, \qquad D:=\sum_v\pi(v)^2.
\]

Kac's formula \cite[Equation~2.24]{aldous-fill-2014} gives the exact normalisation
\begin{equation}
\e_\rho T_{\rm meet}=\frac{1-D}{2A},                              \label{eq:excursion-mean}
\end{equation}
while \cite[Equation~3.21]{coxchenchoi} gives
\begin{equation}
\e_{\pi\otimes\pi}T_{\rm meet}\geq\frac{(1-D)^2}{4A}.           \label{eq:mean-field-lower}
\end{equation}

On $\{T_{\rm mix}<t\}$ the chain after $T_{\rm mix}$ is stationary. Under stationarity, the probability of visiting the diagonal during $[t,2t]$ is at most $D+2At$: the first term covers being on the diagonal at time $t$, and $2A$ is the stationary entrance flux.  Hence
\begin{equation}
\p_\rho(T_{\rm meet}>2t) \geq\p_\rho(T_{\rm mix}<t<T_{\rm meet})-D-2At =\Omega(1).                                                     \label{eq:two-block-survival}
\end{equation}

Here the last equality uses conditions~\ref{cond:clumping-collision}--\ref{cond:clumping-escape}.

The product chain is reversible and has relaxation time of the same order as the one-particle chain.  By the stationary hitting-time approximation
\cite[Proposition~3.23]{aldous-fill-2014}, applied to its diagonal,
\[
\p_{\pi\otimes\pi}(T_{\rm meet}>t) -\p_{\pi\otimes\pi}(T_{\rm meet}>2t) \sim\frac{t}{\e_{\pi\otimes\pi}T_{\rm meet}},
\]
because $t_{\rm rel}\ll t\ll\e_{\pi\otimes\pi}T_{\rm meet}$. On the other hand, the last-exit formula of
\cite[Corollary~3.4]{coxchenchoi} and
\eqref{eq:two-block-survival} give
\[
\p_{\pi\otimes\pi}(T_{\rm meet}>t) -\p_{\pi\otimes\pi}(T_{\rm meet}>2t) =2A\int_t^{2t}\p_\rho(T_{\rm meet}>s)\,\de s \geq2At\p_\rho(T_{\rm meet}>2t)=\Omega(At).
\]

Thus $\e_{\pi\otimes\pi}T_{\rm meet}=O(A^{-1})$, which together with
\eqref{eq:mean-field-lower} proves the claim.
\end{proof}

With this definition of a fast-mixing chain (that the return probability and the stationary hitting probability have the same order), we thus have the stationary meeting time. From there, it is well known we can also say that the full expected coalescence time of $N$ walkers has the same order, as in the following.

\begin{lemma}[Coalescence comparison]                              \label{lem:coalescence-comparison}
There is a universal constant $C$ such that, for every finite reversible continuous-time chain on $N\geq2$ states, with $\pi_*:=\min_v\pi(v)$,
\begin{align}
t_{\rm meet}&\leq t_{\rm meet}^\pi
   +C t_{\rm rel}(1+\log\pi_*^{-1}),                              \label{eq:worst-stationary-meeting}\\
\e T_{\rm coal}&\leq C t_{\rm meet} \left(1+\sqrt{t_{\rm mix}/t_{\rm meet}}\log N\right).           \label{eq:kms-continuous}
\end{align}
\end{lemma}

\begin{proof}
The product chain has relaxation time $t_{\rm rel}$ and minimum stationary mass $\pi_*^2$.  Hence the standard relaxation bound gives an optimal strong stationary time with mean at most $t_{\rm rel}(1+2\log\pi_*^{-1})$; see
\cite[Theorem~1]{fill1991time}.  Ignoring meetings until this time and
then starting from $\pi\otimes\pi$ immediately proves
\eqref{eq:worst-stationary-meeting}.

For \eqref{eq:kms-continuous}, put
\[
q_*:=\max_v q(v),\qquad h:=(2q_*)^{-1},\qquad P_h:=e^{hQ}.
\]

This kernel is reversible and
\[
P_h(v,v)\geq e^{-hq(v)}\geq e^{-1/2}>\frac12,
\]
so the reversible-kernel coalescence estimate
\cite[Lemma~25]{koval2026meeting}, extending
\cite[Theorem~1.1]{kanade23}, applies directly with a universal constant
$C$.  Write a superscript $h$ for the corresponding synchronous discrete chain, with time counted in timesteps.  Since $P_h^k=e^{khQ}$,
\[
t_{\rm mix}\leq ht_{\rm mix}^h\leq t_{\rm mix}+h.
\]

A meeting at a timestep is also a continuous-time meeting, and hence $t_{\rm meet}\leq ht_{\rm meet}^h$.  Conversely, after a continuous-time meeting, neither walk jumps before the next timestep with probability at least $e^{-2q_*h}=e^{-1}$.  Restarting at that timestep when this event fails gives
\[
ht_{\rm meet}^h \leq t_{\rm meet}+h+(1-e^{-1})ht_{\rm meet}^h.
\]

Therefore
\[
t_{\rm meet}\leq ht_{\rm meet}^h\leq e(t_{\rm meet}+h).
\]

If meetings between timesteps are ignored and walkers coalesce only when they occupy the same state at a timestep, coalescence can only be delayed and the resulting process is the $P_h$-coalescent.  Thus
\[
T_{\rm coal}\leq hT_{\rm coal}^h
\]
pathwise.  This is the priority coupling in the proof of
\cite[Lemma~26]{koval2026meeting}.

For distinct initial states a meeting cannot precede the first jump of either walk.  That time has rate at most $2q_*$ and hence mean at least $h$, so $t_{\rm meet}\geq h$.  Choose also $v$ with $\pi(v)\leq1/2$.  The event that the walk at $v$ makes no jump before time $t$ gives
\[
\bigl\|e^{tQ}(v,\cdot)-\pi\bigr\|_{\rm TV} \geq e^{-q_*t}-\frac12,
\]
and therefore $t_{\rm mix}\geq q_*^{-1}\log(4/3)$.  The additive $h$ terms in the preceding comparisons can consequently be absorbed into $t_{\rm meet}$ and $t_{\rm mix}$.  Only universal numerical factors are absorbed into $C$, and the cited estimate now gives
\[
\e T_{\rm coal} \leq C\left(ht_{\rm meet}^h +\sqrt{(ht_{\rm meet}^h)(ht_{\rm mix}^h)}\log N\right) \leq Ct_{\rm meet} \left(1+\sqrt{t_{\rm mix}/t_{\rm meet}}\log N\right),
\]
which is \eqref{eq:kms-continuous}.
\end{proof}

\subsection{The hyperbolic graph}

Throughout, empirical degree moments are taken over the giant component:
\[
M_q:=\sum_{v\in V(\cG)}\de(v)^q, \qquad q\in\mathbb R.
\]

For $\varphi=0$ on the giant component,
\cite{dieter_mixing} gives, with high probability,
\[
t_{\rm rel} =O\left( n^{2\alpha-1} \log^{\frac{1}{1-\alpha}} n \right).
\]

By \cite[Lemma 4.23]{aldous-fill-2014}, the mixing time is at most a logarithmic factor larger.  Thus
\[
t_{\rm mix}(0)=\widetilde O(n^{2\alpha-1})
\]
with high probability, and the same upper bound holds for every fixed $\varphi\geq0$.  Indeed,
\[
\var_\pi(f) = \frac12\sum_i\sum_j (f(i)-f(j))^2\pi(i)\pi(j),
\]
and therefore, for \(\varphi\geq0\),
\[
\begin{split}
t_{\rm rel}(\varphi) &= \frac{1}{M_{1-\varphi}} \sup_f \frac{\sum_i\sum_j (f(i)-f(j))^2 \de(i)^{1-\varphi}\de(j)^{1-\varphi}} {\sum_i\sum_j
 (f(i)-f(j))^2\mathbbm{1}_{i\sim j}}\\
&\leq \frac{M_1}{M_{1-\varphi}}\, t_{\rm rel}(0).
\end{split}
\]

Since $1-\varphi\leq1<2\alpha$, Proposition~\ref{prop:giant-degree-moments} gives $M_1\whpasymp n$ and $M_{1-\varphi}\whpasymp n$ for every fixed $\varphi\geq0$.  Finally, if $d_{\max}:=\max_v\de(v)$, then the largest-degree estimate above gives, with high probability,
\[
\pi_*^{-1}\leq M_{1-\varphi}d_{\max}^{(\varphi-1)_+} \leq n^{1+(\varphi-1)_+/(2\alpha)+o(1)}.
\]

The relaxation-to-mixing comparison therefore proves the following generalisation of \cite{dieter_mixing}.

\begin{corollary}[Mixing bound]\label{cor_comparison}
For every fixed $\varphi\geq0$, with high probability,
\[
t_{\rm mix}(\varphi) =\widetilde O(n^{2\alpha-1}).
\]
\end{corollary}

\section{Proofs of the main theorem}

We first prove the lower bound simultaneously in all three regimes.  Let $I,J$ be independent samples from $\pi$. We can lower bound the consensus event by the event that $I$ and $J$ disagree, giving by voter model duality
\[
\p_{\mu_u}(T_{\rm cons}>t\mid\cG) \geq \p_{\mu_u}(\xi_t(I)\neq\xi_t(J)\mid\cG) = 2u(1-u)\p_{\pi\otimes\pi}(T_{\rm meet}>t\mid\cG).
\]

Integrating in $t$ and applying \cite[Equation~3.21]{coxchenchoi} yields
\begin{equation}
\e_{\mu_u}(T_{\rm cons}\mid\cG) \geq 2u(1-u)t_{\rm meet}^{\pi} \geq \frac{u(1-u)}2 \frac{M_{1-\varphi}^2}{M_{2-\varphi}} \left(1-\sum_v\pi(v)^2\right)^2.                    \label{eq:general-consensus-lower}
\end{equation}

Corollary~\ref{cor:degree-collision} and Proposition~\ref{prop:giant-degree-moments} show that the right-hand side has the order claimed in Theorem \ref{thm:consensus}.  This lower bound is simple because the claimed ``mean-field'' order, also suggested by the heterogeneous mean-field methods of \cite{sood2008voter}, is the fastest possible order for the given $\pi$ and $q$.

It remains to prove the upper bounds.  We first treat $\varphi>2-2\alpha$, using the random-target estimate of
\cite{MR4816414} rather than an escape probability.

\begin{proof}[Upper bound in Theorem~\ref{thm:consensus} for
$\varphi>2-2\alpha$] The chain dual to the voter model has generator
\[
Q_{v,w}=\de(v)^{\varphi-1}\mathbbm{1}_{v\sim w}, \qquad \pi(v)=\frac{\de(v)^{1-\varphi}}{M_{1-\varphi}}.
\]

Consequently its edge conductances are constant:
\begin{equation}
c_\varphi(v,w):=\pi(v)Q_{v,w} =\frac{\mathbbm{1}_{v\sim w}}{M_{1-\varphi}}.                     \label{eq:constant-conductance}
\end{equation}
Write $\cR_{\cG}(x\leftrightarrow y)$ for effective resistance when every edge of $\cG$ has unit resistance.  The continuous-time commute identity and
\eqref{eq:constant-conductance} give
\begin{equation}
\e_x T_y+\e_y T_x =M_{1-\varphi}\cR_{\cG}(x\leftrightarrow y).                     \label{eq:weighted-commute}
\end{equation}
Averaging \eqref{eq:weighted-commute} against $\pi\otimes\pi$ gives the random-target time
\[
t_{\odot}^{(\varphi)}:=\sum_y\pi(y)\e_\pi T_y =\frac1{2M_{1-\varphi}}\sum_{x,y} \de(x)^{1-\varphi}\de(y)^{1-\varphi} \cR_{\cG}(x\leftrightarrow y).
\]

For $\varphi=0$, the walk jumps at rate one and its embedded jump chain is discrete-time simple random walk.  Conditional on the jump chain taking $m$ steps to hit a given target, the continuous hitting time is a sum of $m$ independent mean-one exponential variables, and thus the continuous and discrete random-target times agree exactly.  Hence we can apply \cite[Theorem~1]{MR4816414} to say
\[
t_{\odot}^{(0)} =\frac1{2M_1}\sum_{x,y}\de(x)\de(y) \cR_{\cG}(x\leftrightarrow y) \whpasymp n.
\]

Since $\varphi>0$ we have $\de(v)^{1-\varphi}\leq\de(v)$, and therefore with high probability
\[
t_{\odot}^{(\varphi)} \leq\frac{M_1}{M_{1-\varphi}}t_{\odot}^{(0)}=O(n).
\]

By \cite[Proposition~2]{aldous91}, followed by Jensen's inequality, and the mixing time of Corollary~\ref{cor_comparison} we have, with high probability,
\[
t_{\rm meet}=O\bigl(t_{\rm mix}+t_{\odot}^{(\varphi)}\bigr)=O(n).
\]

Put $N:=|V(\cG)|$, so $N\whpasymp n$.  Lemma~\ref{lem:coalescence-comparison} now yields, with high probability,
\[
\e(T_{\rm coal}\mid\cG) =O\bigl(t_{\rm meet}+\sqrt{t_{\rm meet}t_{\rm mix}}\log N\bigr)=O(n),
\]
because, with high probability, the square-root term is $\widetilde O(n^\alpha)=o(n)$.  Graphical duality consequently gives, with high probability,
\begin{equation}
\e_{\mu_u}(T_{\rm cons}\mid\cG)\leq\e(T_{\rm coal}\mid\cG)=O(n).       \label{eq:large-rate-upper}
\end{equation}

Together with \eqref{eq:general-consensus-lower}, this proves \eqref{eq_theorem_3}.
\end{proof}

In the remaining harder cases, we will construct a flow to lower bound the escape probability by a positive constant.

\begin{theorem}[Escape probability]\label{thm:escape-all}
Suppose $0\leq\varphi\leq 2-2\alpha$.  There exists some $\kappa>0$ such that
\begin{equation}
\p_\rho(T_{\rm mix}<n^\kappa<T_{\rm meet}\mid\cG)\pasymp1.
\end{equation}

At $\varphi=2-2\alpha$, this holds in the stronger sense that there is a fixed $c>0$ such that, with high probability,
\[
\p_\rho(T_{\rm mix}<n^\alpha<T_{\rm meet}\mid\cG)\geq c.
\]
\end{theorem}

The required escape estimates are proved in Theorems~\ref{thm:escape-hrg} and~\ref{thm:critical-escape}, using the two flow constructions developed below.  We first show how these estimates complete the proof of the main theorem.

\begin{proof}[Upper bound in Theorem~\ref{thm:consensus} for
$0\leq\varphi\leq2-2\alpha$] Proposition~\ref{prop:giant-degree-moments} gives
\begin{equation}
\begin{aligned}
\frac{M_{1-\varphi}^2}{M_{2-\varphi}} &\pasymp n^{\,2-1/\alpha+\varphi/(2\alpha)},
  &&0\leq\varphi<2-2\alpha,\\
\frac{M_{1-\varphi}^2}{M_{2-\varphi}} &\whpasymp n/\log n, &&\varphi=2-2\alpha.
\end{aligned}                                                       \label{eq:small-critical-scales}
\end{equation}

We check the conditions of Proposition~\ref{prop:poisson-clumping} with $t=n^\kappa$.  Condition~\ref{cond:clumping-collision} follows from Corollary~\ref{cor:degree-collision}.  For condition~\ref{cond:clumping-window}, in the strictly subcritical case choose $\beta$ and $\kappa$ as in
\eqref{eq:beta-window} and \eqref{eq:kappa-window}.  Then
\[
2\alpha-1<\kappa<1-4\alpha(1-\alpha)\beta <2-\frac1\alpha+\frac{\varphi}{2\alpha},
\]
whereas at criticality in Theorem~\ref{thm:critical-escape} we take $t=n^\alpha$, so with $\kappa=\alpha$ in this case we use $2\alpha-1<\alpha<1$.  Corollary~\ref{cor_comparison} and
\eqref{eq:small-critical-scales} therefore give, with high probability,
$t_{\rm mix}\ll n^\kappa\ll {M_{1-\varphi}^2}/{M_{2-\varphi}}$ in both cases. Condition~\ref{cond:clumping-escape} is precisely Theorem~\ref{thm:escape-all}.  Thus, for every $\varepsilon>0$, Proposition~\ref{prop:poisson-clumping} gives, outside a set of graphs of probability at most $\varepsilon+o(1)$,
\[
t_{\rm meet}^\pi=O_\varepsilon\left(\frac{M_{1-\varphi}^2}{M_{2-\varphi}}\right).
\]

It remains to check that the two error terms in Lemma~\ref{lem:coalescence-comparison} are negligible. Since $\varphi\in [0,1)$, Proposition~\ref{prop:giant-degree-moments} gives, with high probability,
\[
\pi_*^{-1}\leq M_{1-\varphi}\leq M_{1}=O(n).
\]

The same polynomial gaps and Corollary~\ref{cor_comparison} give, with high probability,
\begin{equation}
t_{\rm rel}(1+\log\pi_*^{-1}) =o\left(\frac{M_{1-\varphi}^2}{M_{2-\varphi}}\right), \qquad t_{\rm mix}\log^2N =o\left(\frac{M_{1-\varphi}^2}{M_{2-\varphi}}\right).     \label{eq:coalescence-gap}
\end{equation}

After enlarging the same exceptional set by $o(1)$,
\eqref{eq:worst-stationary-meeting} gives
$t_{\rm meet}=O_\varepsilon(M_{1-\varphi}^2/M_{2-\varphi})$. Hence the second relation in \eqref{eq:coalescence-gap} gives
\[
\sqrt{t_{\rm meet}t_{\rm mix}}\log N =o\left(\frac{M_{1-\varphi}^2}{M_{2-\varphi}}\right),
\]
and \eqref{eq:kms-continuous} and duality give
\[
\e_{\mu_u}(T_{\rm cons}\mid\cG) \leq\e(T_{\rm coal}\mid\cG) =O_\varepsilon\left(\frac{M_{1-\varphi}^2}{M_{2-\varphi}}\right)
\]
outside a set of graphs of probability at most $\varepsilon+o(1)$. At criticality, the stronger assertion in Theorem~\ref{thm:escape-all} and the same negligibility estimates give, with high probability,
\[
\e_{\mu_u}(T_{\rm cons}\mid\cG) =O\left(\frac{M_{1-\varphi}^2}{M_{2-\varphi}}\right).
\]

The orders
\eqref{eq:small-critical-scales} make this prove the upper bounds of \eqref{eq_theorem_1} and \eqref{eq_theorem_2}.
\end{proof}

Recall that the lower bounds were immediate in \eqref{eq:general-consensus-lower}. It remains to find this $\Omega(1)$ escape probability, i.e. to prove Theorem \ref{thm:escape-all}.

\subsection{Strictly subcritical escape}

We now verify condition~\ref{cond:clumping-escape} of Proposition~\ref{prop:poisson-clumping} for the range
\[
0\leq\varphi<2-2\alpha.
\]

Put also $\chi:=4\alpha(1-\alpha)$. We have two alternative definitions for the one-dimensional and the two-dimensional flows.  A base-graph flow is an antisymmetric function on the edges of the graph, with
\[
\operatorname{div}_{\cG}f(x):=\sum_{x'\sim x}f(x,x'), \qquad \cE_{\cG}(f):=\frac12\sum_x\sum_{x'\sim x}f(x,x')^2.
\]

Thus $\cE_{\cG}$ is the energy for unit edge conductances. For the degree-dependent walk in \eqref{eq:voter-kernel}, we normalise so that the edge conductance is equal to the ergodic rate of traversing the (directed) edge with edge conductance $1/M_{1-\varphi}$ and hence flow energy $M_{1-\varphi}\cE_{\cG}(f)$.

For the product chain, we have to consider the conductances associated to the product Markov chain and in this case that network does not have constant conductance. Put
\begin{align*}
c_\times((x,y),(x',y)) &=\frac{\de(y)^{1-\varphi}}{M_{1-\varphi}^2},
 &&xx'\in E(\cG),\\
c_\times((x,y),(x,y')) &=\frac{\de(x)^{1-\varphi}}{M_{1-\varphi}^2}, &&yy'\in E(\cG),
\end{align*}
and, for an antisymmetric product flow $F$, set
\[
\operatorname{div}_\times F(z):=\sum_{z'\sim_\times z}F(z,z'), \qquad \cE_\times(F):=\frac12\sum_z\sum_{z'\sim_\times z} \frac{F(z,z')^2}{c_\times(z,z')}.
\]

\begin{definition}[Lifted flow]
The first-coordinate and  second-coordinate lifts of a flow $f$ and a measure $\lambda$ are respectively
\[
\mathsf L_1(f,\lambda)((x,y),(x',y'))=
 \begin{cases}
 f(x,x')\lambda(y), & y=y',\\
0, & \text{otherwise}
 \end{cases}
\]
and
\[
\mathsf L_2(\lambda,f)((x,y),(x',y'))=
 \begin{cases}
 f(y,y')\lambda(x), & x=x',\\
0, & \text{otherwise.}
 \end{cases}
\]
\end{definition}

\begin{lemma}[Energy of a lifted flow]                          \label{lem:weighted-lift}
The lifted flow has divergence and energy
\begin{equation}
\operatorname{div}_\times\mathsf L_1(f,\lambda) =(\operatorname{div}_{\cG}f)\otimes\lambda, \qquad \cE_\times(\mathsf L_1(f,\lambda)) =M_{1-\varphi}^2I_\varphi(\lambda)\cE_{\cG}(f), \label{eq:weighted-lift}
\end{equation}
where for a probability measure $\lambda$ on $V(\cG)$ we write
\[
I_\varphi(\lambda):= \sum_y\frac{\lambda(y)^2}{\de(y)^{1-\varphi}}.
\]

The second-coordinate lift has divergence $\lambda\otimes\operatorname{div}_{\cG}f$ and the same energy.
\end{lemma}

\begin{proof}
Using the definitions of both energies,
\[
\begin{split}
\cE_\times(\mathsf L_1(f,\lambda)) &=\frac12\sum_y\sum_x\sum_{x'\sim x} \frac{f(x,x')^2\lambda(y)^2}
        {\de(y)^{1-\varphi}/M_{1-\varphi}^2}\\
&=M_{1-\varphi}^2 \left(\sum_y\frac{\lambda(y)^2}{\de(y)^{1-\varphi}}\right) \left(\frac12\sum_x\sum_{x'\sim x}f(x,x')^2\right),
\end{split}
\]
which is \eqref{eq:weighted-lift}.  The calculation for the second-coordinate lift is identical.
\end{proof}

We next record the part of the tiling of
\cite[Section~4.1]{MR4816414} used in the two flow lemmas below.  For
positive radial coordinates $a,b$, let $\theta_R(a,b)$ be the largest angular separation of two such points whose hyperbolic distance is $R$ (i.e. if the points are within this angular separation, we know they are adjacent).  This is
\[
\theta_R(a,b):=
 \begin{cases}
\pi, & a+b\leq R,\\[2mm] \displaystyle\arccos\!\left( \frac{\cosh a\cosh b-\cosh R}{\sinh a\sinh b}\right), & a+b>R.
 \end{cases}
\]

Fix a sufficiently large tiling parameter $c_{\rm t}>0$ and set
\begin{equation}
\begin{gathered}
r_{-1}^{\rm t}:=0,\qquad r_0^{\rm t}:=R/2,\qquad r_1^{\rm t}:=(R+c_{\rm t})/2,\\[0.5em] r_j^{\rm t}:=\sup\left\{r:\theta_R(r,r)\geq \tfrac12\theta_R(r_{j-1}^{\rm t},r_{j-1}^{\rm t})\right\}, \qquad j\geq2,\\[0.2em] N_0:=\left\lceil\frac{2\pi} {\theta_R(r_1^{\rm t},r_1^{\rm t})}\right\rceil, \qquad \theta_j:=\frac{2\pi}{2^jN_0}.
\end{gathered}                                                    \label{eq:half-tile-levels}
\end{equation}

Partition the angular circle into $N_0$ half-open intervals of length $\theta_0$, starting at angle $0$, and obtain the level-$j$ intervals by bisecting at each successive level.  A level-$j$ tile is a set
\[
\{(r,\vartheta):r_{j-1}^{\rm t}\leq r<r_j^{\rm t},\ \vartheta\in I\},
\]
where $I$ is one of the level-$j$ intervals.  Its two \emph{half-tiles} are obtained by bisecting $I$.  If $A$ is a level-$j$ half-tile, its angular interval is one of the level-$(j+1)$ intervals; the two half-tiles of the corresponding level-$(j+1)$ tile are the children of $A$.  Thus the $k$th generation below $A$ consists of $2^k$ half-tiles whose angular intervals partition that of $A$.  By
\cite[Remark~21]{MR4816414}, if $A'$ is a child of $A$, then
\begin{equation}
d_{\mathbb H}(x,y)\leq R \qquad\text{for every }x\in A\text{ and }y\in A'.             \label{eq:half-tile-adjacency}
\end{equation}

Put $h_j:=R-r_j^{\rm t}$.  The estimates in
\cite[Claims~23--24]{MR4816414} give, uniformly over the levels used
below,
\begin{equation}
h_j-h_{j+k}=k\log2+O(1),\qquad \frac{\theta_j}{2}\asymp e^{h_j-R/2},\qquad \e|\mathcal P_n\cap A|\asymp e^{(1-\alpha)h_j}                 \label{eq:half-tile-scales}
\end{equation}
for every level-$j$ half-tile $A$.  Call $A$ \emph{occupied} when $\mathcal P_n\cap A\ne\varnothing$, and let $\mathrm{unif}_{\mathcal P_n\cap A}$ denote the uniform probability measure on $\mathcal P_n\cap A$.

\begin{lemma}[Flow through one initial half-tile]
\label{lem:single-descendant-flow}
Fix a sufficiently large constant $h_*>0$ and a terminal level $j_*$ with $h_{j_*}=h_*+O(1)$.  Let $H$ be a level-$j$ half-tile with $j\leq j_*$, and put
\[
k:=j_*-j,\qquad m:=\exp\{(1-\alpha)(h_j-h_{j_*})\}.
\]

Suppose that a prescribed point $v$ lies outside $H$ and all its descendants through level $j_*$, and is adjacent to every point of $H$. Under the Palm law obtained by adding $v$, with asymptotically positive probability one can construct a collection $\mathcal L_H$ of occupied level-$j_*$ descendants of $H$ and a unit flow $\Gamma_H$ on the component of $v$, supported on $v$ and the points in the descendant half-tiles between levels $j$ and $j_*$.  The collection contains at least a fixed positive proportion of the $2^k$ terminal descendants and, with $Z_H:=|\mathcal L_H|$,
\[
\operatorname{div}\Gamma_H=\mathbbm{1}_v- \frac1{Z_H}\sum_{L\in\mathcal L_H} \mathrm{unif}_{\mathcal P_n\cap L},
\]
while $\cE_{\cG}(\Gamma_H)\leq C/m$ for a fixed $C<\infty$.

The construction uses only the Poisson points in these half-tiles. Success events for initial half-tiles with disjoint descendant regions are therefore independent.
\end{lemma}

\begin{proof}
A level-$j_*$ descendant $L$ is \emph{live} if every half-tile on the unique path from $H$ to $L$ is occupied.  Let $\mathcal L_H$ be the set of live terminal descendants and, for a half-tile $A$ at or below $H$, put
\[
Z_H(A):=|\{L\in\mathcal L_H:L\text{ descends from }A\}|, \qquad Z_H:=Z_H(H).
\]

At depth $q$ below $H$, \eqref{eq:half-tile-scales} shows that each descendant half-tile has mean population $\asymp m2^{-(1-\alpha)q}$.  Read a line of descent from its terminal leaf inwards.  By taking $h_*$ large, its successive occupation probabilities are bounded below by $1-\exp\{-c2^{(1-\alpha)t}\}$, $t\geq0$, for a fixed $c>0$. These half-tiles occupy disjoint radial annuli, so the events are independent.  Hence
\[
p_*:=\prod_{t=0}^\infty \left(1-\exp\{-c2^{(1-\alpha)t}\}\right)>0
\]
is a uniform lower bound on the probability that any prescribed terminal line is live.  Thus $\e Z_H\geq p_*2^k$.  Since $0\leq Z_H\leq2^k$,
\[
\e Z_H^2\leq2^k\e Z_H.
\]
Since $p_*2^{k-1}\leq\e Z_H/2$, Paley--Zygmund gives
\begin{equation}
\p\left(Z_H\geq p_*2^{k-1}\right) \geq\p\left(Z_H\geq\frac12\e Z_H\right) \geq\frac{(\e Z_H)^2}{4\e Z_H^2} \geq\frac{p_*}{4}.                                        \label{eq:single-tile-pz}
\end{equation}

On the event $Z_H>0$, orient the half-tile tree away from $H$ and put
\[
\Gamma_H(v,x)=\frac1{|\mathcal P_n\cap H|}, \qquad x\in\mathcal P_n\cap H,
\]
and, whenever $A'$ is a child of $A$ with $Z_H(A')>0$, put
\[
\Gamma_H(x,y)=\frac{Z_H(A')} {Z_H|\mathcal P_n\cap A|\,|\mathcal P_n\cap A'|}, \qquad x\in\mathcal P_n\cap A,\ y\in\mathcal P_n\cap A'.
\]

Extend $\Gamma_H$ antisymmetrically and set it to zero on all remaining edges.  The source edges exist by hypothesis and the remaining edges by
\eqref{eq:half-tile-adjacency}.  For $x\in\mathcal P_n\cap A$, the current
entering from the preceding generation is $Z_H(A)/(Z_H|\mathcal P_n\cap A|)$, and the current leaving through the two children is
\[
\sum_{A'\text{ child of }A}\sum_{y\in\mathcal P_n\cap A'}\Gamma_H(x,y) =\frac{\sum_{A'}Z_H(A')}{Z_H|\mathcal P_n\cap A|} =\frac{Z_H(A)}{Z_H|\mathcal P_n\cap A|}.
\]

This agrees with the current arriving from $v$ at $H$, while a vertex $x\in\mathcal P_n\cap L$ in a live terminal descendant receives $1/(Z_H|\mathcal P_n\cap L|)$.  This proves the asserted divergence, and every vertex used is connected to $v$.

On the event in \eqref{eq:single-tile-pz}, the total current entering a depth-$q$ half-tile is at most $2^{1-q}/p_*$.  The complete-bipartite energy formula and
\[
\e[\mathbbm1_{\{N>0\}}/N] \leq 2\min\{\e N,(\e N)^{-1}\}
\]
for Poisson $N$ show that at each nonterminal depth $q$ the expected energy contribution, restricted to this event, is $ O\left(m^{-2}2^{-(2\alpha-1)q}\right). $

The source edges have energy
\[
\sum_{x\in\mathcal P_n\cap H}\Gamma_H(v,x)^2 =\frac1{|\mathcal P_n\cap H|},
\]
whose expected contribution on this event is $O(m^{-1})$.  At each of the bounded number of depths nearest level $j_*$, there are $O(2^k)$ half-tiles, each receiving total current $O(2^{-k})$.  Their combined energy is therefore at most $ O(2^k)O(2^{-2k})=O(2^{-k}). $

Thus
\begin{align*}
&\e\left[ \mathbbm1_{\{Z_H\geq p_*2^{k-1}\}} \frac12\sum_x\sum_{y\sim x}\Gamma_H(x,y)^2
 \right] \\
&\qquad=O\left( \frac1m+\frac1{m^2}\sum_{q\geq0}2^{-(2\alpha-1)q}+2^{-k} \right) =O(m^{-1}).
\end{align*}

Indeed, the series converges because $\alpha>1/2$, and $m\asymp2^{(1-\alpha)k}$ gives
\[
m^{-2}=O(m^{-1}),\qquad 2^{-k}=O(m^{-1}).
\]

Choose the fixed $C$ in the statement large enough that the expectation above is at most $p_*C/(8m)$.  Markov's inequality and
\eqref{eq:single-tile-pz} then give
\begin{align*}
&\p\left(Z_H\geq p_*2^{k-1},\quad
 \frac12\sum_x\sum_{y\sim x}\Gamma_H(x,y)^2\leq\frac Cm\right) \\
&\qquad\geq \p\left(Z_H\geq p_*2^{k-1}\right) -\frac mC\e\left[ \mathbbm1_{\{Z_H\geq p_*2^{k-1}\}} \frac12\sum_x\sum_{y\sim x}\Gamma_H(x,y)^2
 \right] \\
&\qquad\geq\frac{p_*}{4}-\frac{p_*}{8} =\frac{p_*}{8}.
\end{align*}

This proves the joint positive-probability assertion.  All definitions use only the Poisson points in the stated half-tiles, as asserted.
\end{proof}

The picture behind the next lemma is simple.  At each scale we choose many disjoint half-tiles whose points are adjacent to $v$.  From each successful half-tile the flow branches through occupied descendants to a fixed terminal level near the boundary, and we average the resulting flows.  The disjoint regions succeed independently; averaging lowers the energy, while spreading the terminal measure over many descendants lowers $I_\varphi$.

We do this at several initial levels and on both sides of $v$, choosing the angular intervals so that different flows meet only at $v$.  These are flows on $\cG$, but they come in two families because the final flow lives on $\cG\times\cG$: we move the first coordinate using one family and then the second coordinate using the other.  Their separated supports will in the end keep this product flow off the meeting set.

\begin{lemma}[Separated multiscale flows]                         \label{lem:separated-multiscale-flows}
Fix $\varphi<1$, $\epsilon_0>0$, a sufficiently small fixed $\epsilon>0$ and a sufficiently large fixed $\ell$.  Let $v$ be a prescribed point with
\[
h(v):=R-r_v\geq\left(\frac{1}{2}+\epsilon_0\right)R,
\]
and write $\vartheta(v)$ for its angular coordinate.  Put
\[
J:=\left\lfloor\frac{R-h(v)}{\ell}\right\rfloor, \qquad u_i:=i\ell,\quad 0\leq i\leq J.
\]

For each $i$ and $b\in\{0,1\}$, the construction uses $K_i$ initial half-tiles, where
\begin{equation}
K_i\asymp e^{h(v)/2}e^{-u_i/2}.                              \label{eq:initial-half-tile-count}
\end{equation}
There is a fixed $c>0$ such that, uniformly over the choices above, under the Palm law obtained by adding $v$ the construction for $(i,b)$ fails with probability at most $e^{-cK_i}$.  On success, whenever $v\in V(\cG)$, it produces a probability measure $\lambda_i^b$ and a unit flow $g_i^b$ on $\cG$ such that
\begin{equation}
\operatorname{div}g_i^b=\mathbbm{1}_v-\lambda_i^b,\qquad \cE_{\cG}(g_i^b)\leq\frac{C}{P_i},\qquad I_\varphi(\lambda_i^b)\leq\frac{C}{V_i},                         \label{eq:weighted-flow-families}
\end{equation}
where
\begin{equation}
P_i\asymp e^{h(v)/2}e^{-(2\alpha-1)u_i/2}, \qquad V_i\asymp e^{h(v)/2}e^{u_i/2}.                                 \label{eq:flow-family-scales}
\end{equation}

Let $Q_i^b$ be the support of $\lambda_i^b$, and let $S_i^b$ consist of $v$ and the endpoints of edges carrying nonzero $g_i^b$-flow.  Then $Q_i^b\subset S_i^b\setminus\{v\}$, every vertex in $Q_i^b$ has height $O(1)$ uniformly in $i$ and $b$, and every $x\in S_i^b\setminus\{v\}$ satisfies
\[
\operatorname{dist}_{\mathbb S^1}(\vartheta(x),\vartheta(v))\leq \epsilon,
\]
and
\[
S_i^b\cap S_j^{b'}=\{v\}\qquad\text{whenever }(i,b)\ne(j,b').
\]
\end{lemma}

\begin{proof}
Use the half-tile hierarchy specified by \eqref{eq:half-tile-levels} and the subsequent interval-bisection rule, rotated through $\vartheta(v)$, and use signed angular coordinates centred at $\vartheta(v)$, so that $v$ has angle $0$.  Choose $h_*$ sufficiently large for Lemma~\ref{lem:single-descendant-flow}.  Choose a terminal level $j_*$ with $h_{j_*}=h_*+O(1)$ and, for each $i$, an initial level $j_i\leq j_*$ with
\[
h_{j_i}=h_*+u_i+O(1).
\]
These levels exist because $u_i\leq R-h(v)\leq(1/2-\epsilon_0)R$.  If $k_i:=j_*-j_i$, then \eqref{eq:half-tile-scales} gives
\begin{equation}
k_i=\frac{u_i}{\log2}+O(1),\qquad 2^{k_i}\asymp e^{u_i}.        \label{eq:descendant-depth}
\end{equation}

A terminal descendant of a level-$j_i$ half-tile will always mean one of its $2^{k_i}$ descendants at level $j_*$.

Define
\[
\rho_i:=\epsilon e^{(h(v)+u_i-R)/2};\qquad \rho_{-1}:=0;
\]
\[
\mathcal A_i^0:=(8\rho_{i-1},\rho_i/8);\qquad \mathcal A_i^1:=(-\rho_i/8,-8\rho_{i-1}).
\]

Since $u_i\leq R-h(v)$, choosing $\epsilon$ sufficiently small gives $\rho_i\leq \epsilon<\pi/4$; \cite[Lemma~10]{MR4816414} then shows that every point of a level-$j_i$ half-tile whose angular interval is contained in $\mathcal A_i^b$ is adjacent to $v$.

A level-$j_i$ half-tile has angular width $\asymp e^{h_*+u_i-R/2}$, whereas $|\mathcal A_i^b|\asymp\rho_i$.  Hence, uniformly in $b$, the number of such half-tiles is $\asymp \rho_i/e^{h_*+u_i-R/2} \asymp e^{h(v)/2}e^{-u_i/2}$. These counts are smallest at the last scale and are at least $e^{\epsilon_0R+O(1)}$.  Let $K_i$ be the smaller of the two counts and retain $K_i$ of these half-tiles on each side.  This gives
\eqref{eq:initial-half-tile-count}.  Put
\[
m_i:=e^{(1-\alpha)u_i},\qquad P_i:=K_im_i,\qquad V_i:=K_i2^{k_i}.
\]

Equations \eqref{eq:initial-half-tile-count} and
\eqref{eq:descendant-depth} give \eqref{eq:flow-family-scales}.  Every
descendant stays in the angular interval of its initial half-tile, so the descendant regions retained for distinct pairs $(i,b)$ are disjoint, and none contains $v$.

Apply Lemma~\ref{lem:single-descendant-flow} to each retained initial half-tile; the quantity $m$ in that lemma is comparable to $m_i$.  The success events are independent because the descendant regions are disjoint.  Hence, except with probability $e^{-cK_i}$, Chernoff's inequality gives a family $\mathscr I_i^b$ of size $\Omega(K_i)$ successful initial half-tiles.  For each of them use the objects $\mathcal L_H,Z_H,\Gamma_H$ from that lemma.

Average these flows by putting
\[
\mathcal Z_i^b:=\sum_{H\in\mathscr I_i^b}Z_H,\qquad g_i^b:=\sum_{H\in\mathscr I_i^b} \frac{Z_H}{\mathcal Z_i^b}\Gamma_H,
\]
\[
\lambda_i^b:= \frac1{\mathcal Z_i^b} \sum_{H\in\mathscr I_i^b}\sum_{L\in\mathcal L_H} \mathrm{unif}_{\mathcal P_n\cap L}.
\]

Since $Z_H$ is bounded below by a fixed positive multiple of $2^{k_i}$ and above by $2^{k_i}$, $\mathcal Z_i^b\asymp K_i2^{k_i}\asymp V_i$, and the divergence in
\eqref{eq:weighted-flow-families} follows from
Lemma~\ref{lem:single-descendant-flow}.  The edge supports of the $\Gamma_H$ are disjoint, so
\[
\cE_{\cG}(g_i^b) \leq\frac{C}{m_i}\max_{H\in\mathscr I_i^b} \frac{Z_H}{\mathcal Z_i^b} \leq\frac{C}{K_im_i}=\frac{C}{P_i}.
\]

The terminal half-tiles are disjoint and every one is occupied.  Since $\varphi<1$,
\[
I_\varphi(\lambda_i^b) \leq\frac1{(\mathcal Z_i^b)^2} \sum_{H\in\mathscr I_i^b}\sum_{L\in\mathcal L_H} \frac1{|\mathcal P_n\cap L|} \leq\frac1{\mathcal Z_i^b}\leq\frac{C}{V_i}.
\]

The definition of $\lambda_i^b$ also gives $Q_i^b\subset S_i^b\setminus\{v\}$.  All terminal descendants lie at level $j_*$, which is the fixed $O(1)$-height band in the statement.  Every vertex used by $g_i^b$, apart from $v$, has angular coordinate in $\vartheta(v)+\mathcal A_i^b$.  Since $\rho_i\leq \epsilon$, this proves the angular localisation in the statement; the disjointness of the arcs proves the asserted separation of the sets $S_i^b$ and completes the proof.
\end{proof}

We shall also use the following uniform form of the point-to-point flow in
\cite[Proposition~63]{MR4816414}.

\begin{lemma}[Uniform endpoint flows]                              \label{lem:uniform-endpoint-flow}
Fix $\gamma,\zeta>0$.  With high probability, simultaneously for every pair $x,y$ with
\[
h(x),h(y)\geq\gamma R, \qquad \operatorname{dist}_{\mathbb S^1}(\vartheta(x),\vartheta(y))\geq\zeta,
\]
there is a unit flow $h_{x,y}$ on $\cG$ with divergence $\mathbbm{1}_x-\mathbbm{1}_y$ and
\begin{equation}
\cE_{\cG}(h_{x,y})\leq n^{o(1)}\left(\de(x)^{-\chi}+\de(y)^{-\chi}\right).            \label{eq:uniform-endpoint-flow}
\end{equation}
where we recall that $\chi=4\alpha(1-\alpha)$. For some $c_\gamma>0$, apart from its endpoints, the flow may be chosen to use only vertices of height at least $c_\gamma R$; in particular it is disjoint from the fixed outer height band in Lemma~\ref{lem:separated-multiscale-flows}.
\end{lemma}

\begin{proof}
For endpoints in the radial range of \cite[Proposition~63]{MR4816414}, take the flow
\[
f_{x,H_x}+f_{H_x,H_y}-f_{y,H_y}
\]
constructed in its proof.  Here the two endpoint-to-half-tile flows are those of \cite[Lemma~62]{MR4816414}, and the balanced half-tile-to-half-tile flow is the one in \cite[Remark~30]{MR4816414}.  The energy calculation in that proof gives \eqref{eq:uniform-endpoint-flow} in this range.  The maximum-height bound in the proof of Proposition~\ref{prop:giant-degree-moments} gives $h(x),h(y)\leq R/(2\alpha)+o(R)$ uniformly.  If, for example, $h(x)>R/(2\alpha)$, let $x^*$ be the point on the same ray with height $R/(2\alpha)-1$.  Use the endpoint flow defined for $x^*$, replacing $x^*$ by $x$ on its source edges.  The angular monotonicity in
\cite[Remark~9]{MR4816414} shows that all these edges are present, while
Lemma~\ref{lem:uniform-degree-height} shows that the $o(R)$ difference changes the energy by only the displayed $n^{o(1)}$ factor.  The same argument applies to $y$.

In the notation of \cite[Proposition~63]{MR4816414}, the $\ell+1$ truncation in
\cite[equation~(51)]{MR4816414} is inactive when $h(x)\geq\gamma R$.
That equation and \cite[Claim~23]{MR4816414}, also including the comparison construction just described, give
\[
R-h_{\ell'_x} =(2\alpha-1)\min\{h(x),R/(2\alpha)\}+O(1).
\]

The flow $f_{x,H_x}$ uses vertices at level $\ell'_x$ and then parent half-tiles, while $f_{H_x,H_y}$ uses ancestor half-tiles.  Thus every nonendpoint vertex on the $x$ side has at least the height in the last display, and the analogous bound holds on the $y$ side.  This is at least $c_\gamma R$.

Proposition~63 has failure probability $o(n^{-d})$ for arbitrary $d$. Taking $d>2$, Mecke's formula and a union bound over endpoint pairs give the simultaneous assertion.
\end{proof}

We now combine Lemmas~\ref{lem:separated-multiscale-flows} and~\ref{lem:uniform-endpoint-flow} in the product network.

Fix $0\leq\varphi<2-2\alpha$ and
\begin{equation}
\frac{2(1-\alpha)-\varphi}{8\alpha^2(1-\alpha)} <\beta<\frac1{2\alpha}.                                      \label{eq:beta-window}
\end{equation}

Write $h(v):=R-r_v$ and $\vartheta(v)$ for the height and angular coordinate of $v$, and put $\eta:=(\log n)^{-3}$.  Using a fixed tie-breaking rule, choose among the vertices satisfying
\[
|h(w)-\beta R|\leq1, \qquad \operatorname{dist}_{\mathbb S^1}(\vartheta(w),0)\leq\eta
\]
a vertex $w$ for which the endpoint flow of
\cite[Lemma~62]{MR4816414} and the sector edge-cut estimate of
\cite[Lemma~67]{MR4816414} hold; if there is no such choice, use the
first vertex of $\cG$ in the same order.  Put
\begin{equation}
B_w:=\bigl(\{w\}\times(V\setminus\{w\})\bigr) \cup\bigl((V\setminus\{w\})\times\{w\}\bigr),              \label{eq:product-target-set}
\end{equation}
and let $T_w^\times$ be its hitting time.  For a sufficiently large fixed $A_0$, also put
\[
a_v:=\frac{\de(v)^{2-\varphi}}{M_{2-\varphi}}, \qquad \mathcal W_n:=\left\{v: \left|h(v)-\frac{R}{2\alpha}\right|\leq A_0\log\log n\right\}.
\]

For the degree--dependent walk,
\[
\pi(v)=\frac{\de(v)^{1-\varphi}}{M_{1-\varphi}}, \qquad Q_{v,u}=\de(v)^{\varphi-1}\mathbbm{1}_{v\sim u},
\]
and hence
\[
\rho(v,u)=\frac{\mathbbm{1}_{v\sim u} \bigl(\de(v)^{1-\varphi}+\de(u)^{1-\varphi}\bigr)} {2M_{2-\varphi}}.
\]

Thus $\sum_v a_v=1$.  In the product network, identify all vertices of $\Delta:=\{(v,v):v\in V(\cG)\}$ to a single vertex, still denoted $\Delta$.  The total conductance leaving it is
\begin{equation}
\cC_\times(\Delta) =\sum_{x\in V}\pi(x)^2\,2q(x) =\frac{2M_{2-\varphi}}{M_{1-\varphi}^2}.                    \label{eq:shorted-diagonal-c}
\end{equation}

Its normalised exit law chooses an oriented edge $(v,u)$ with probability $a_v/\de(v)$ and then exchanges its coordinates with probability $1/2$; it is therefore exactly $\rho$.  The equilibrium-measure identity gives
\begin{equation}
\p_\rho(T_w^\times<T_{\rm meet}\mid\cG) =\frac{\cC_\times(\Delta\leftrightarrow B_w)} {\cC_\times(\Delta)}.                                      \label{eq:shorted-conductance-identity}
\end{equation}

\begin{proposition}[Target before remeeting]                 \label{prop:escape-to-target}
Under \eqref{eq:beta-window}, the target $w$ above satisfies
\begin{equation}
\p_\rho(T_w^\times<T_{\rm meet}\mid\cG)\pasymp1.          \label{eq:target-before-remeeting}
\end{equation}
\end{proposition}

\begin{proof}
Poisson concentration gives $n^{1-2\alpha\beta+o(1)}$ candidates for $w$.  On the common non-faulty half-tile event of
\cite[Lemma~26]{MR4816414}, the endpoint flow is available for every
candidate.  Mecke's formula applied to Lemma~67 shows that the expected number for which its edge-cut estimate fails is $o(n^{1-2\alpha\beta})$.  Thus a valid choice exists with high probability and
\begin{equation}
\de(w)=n^{\beta+o(1)}.                                      \label{eq:target-degree}
\end{equation}

In particular, $w$ is a function of $\cG$ alone and is fixed before the exit state is sampled from $\rho$.

Since $2-\varphi>2\alpha$, Proposition~\ref{prop:giant-degree-moments} gives
\[
\sum_v a_v^2 =\frac{M_{4-2\varphi}}{M_{2-\varphi}^2}\pasymp1.
\]

Lemma~\ref{lem:extreme-source-window} gives, with high probability and uniformly for $v\in\mathcal W_n$,
\begin{equation}
\sum_{x\notin\mathcal W_n}a_x=o(1), \qquad |\mathcal W_n|\leq(\log n)^{O(1)}, \qquad \de(v)\asymp e^{h(v)/2}=n^{1/(2\alpha)}(\log n)^{O(1)}.
\label{eq:extreme-source-mass}
\end{equation}

We now choose a source at the fixed positive angular separation required by the endpoint flow.  Fix $\varepsilon>0$ and choose $h_\varepsilon>0$ so that $\p(\sum_v a_v^2<h_\varepsilon)\leq\varepsilon/4+o(1)$. Rotational invariance and $\sum_v a_v^2\leq1$ give, for every fixed $0<\delta<\pi/5$,
\[
\e\!\left[\sum_{v:\,\operatorname{dist}_{\mathbb S^1} (\vartheta(v),0)\leq5\delta}a_v^2\right] =\frac{5\delta}{\pi}\e\sum_v a_v^2\leq\frac{5\delta}{\pi}.
\]

Choose a fixed $\delta>0$ so small that the sum is at most $h_\varepsilon/4$ outside an event of probability at most $\varepsilon/4+o(1)$.  Among the vertices in $\mathcal W_n$ outside this arc, choose one, denoted by $s$, with maximal $a_s$.  The squared $a$-mass outside $\mathcal W_n$ is at most the $a$-mass there, and the squared $a$-mass of any set is at most its largest weight.  Hence, outside an event of probability at most $\varepsilon+o(1)$,
\begin{equation}
a_s\geq\frac12h_\varepsilon, \qquad \de(s)=n^{1/(2\alpha)}(\log n)^{O(1)}, \qquad \operatorname{dist}_{\mathbb S^1}(\vartheta(s),\vartheta(w)) \geq4\delta.                                         \label{eq:fixed-separated-source}
\end{equation}

\emph{The product flow to the fixed target.} We use the coordinate lifts and the functional $I_\varphi$ from Lemma~\ref{lem:weighted-lift}.

Fix the auxiliary source $s$ and the target $w$ chosen above.  Their angular distance is at least $4\delta$, fixed independently of
$n$.  %

Since $h(s)=R/(2\alpha)+O(\log\log n)$, it satisfies the height hypothesis of Lemma~\ref{lem:separated-multiscale-flows} for a fixed $\epsilon_0>0$. Apply that lemma with $v=s$ and $\epsilon<\delta$.  At the final scale,
\[
K_J\asymp\frac{\de(s)^2}{n} =n^{(1-\alpha)/\alpha}(\log n)^{O(1)}.
\]

Since $K_i\geq K_J$, for this prescribed source the probability that any scale or either value of $b$ fails is at most $2(J+1)e^{-cK_J}$. Mecke's formula and a union bound over the polylogarithmically many vertices in $\mathcal W_n$, all scales and both values of $b$ show that, with high probability, all the pairs $(g_i^b,\lambda_i^b)$ required below are available.  In particular, for $u_i=i\ell$,
\[
P_i\asymp \de(s)e^{-(2\alpha-1)u_i/2}, \qquad V_i\asymp \de(s)e^{u_i/2}.
\]

Put $\lambda_{-1}^0=\lambda_{-1}^1=\mathbbm{1}_s$, $g_{-1}^0=g_{-1}^1=0$, use the conventions $Q_{-1}^b=S_{-1}^b=\{s\}$, and set $f_i^b=g_i^b-g_{i-1}^b$. The alternating product flow is
\begin{equation}
\Theta_Q:=\sum_{i=0}^J\left\{ \mathsf L_1(f_i^0,\lambda_{i-1}^1) +\mathsf L_2(\lambda_i^0,f_i^1)\right\}.                    \label{eq:alternating-product-flow}
\end{equation}

Thus the other coordinate is frozen alternately at $\lambda_{i-1}^1$ and $\lambda_i^0$, and the scale-$i$ divergence is
\[
(\lambda_{i-1}^0-\lambda_i^0)\otimes\lambda_{i-1}^1 +\lambda_i^0\otimes(\lambda_{i-1}^1-\lambda_i^1).
\]

Indeed, in the first lift the moving coordinate is supported on $S_i^0\cup S_{i-1}^0$ and the frozen coordinate on $Q_{i-1}^1\subseteq S_{i-1}^1$; in the second, the frozen coordinate is in $Q_i^0\subseteq S_i^0$ and the moving coordinate is supported on $S_i^1\cup S_{i-1}^1$.  The intersections in question are empty because the sets $Q_i^b$ do not contain $s$ and distinct $S$-sets intersect only at $s$.  The sole exception is the first lift, for which $\lambda_{-1}^1=\mathbbm{1}_s$ and the flow starts at $(s,s)$.

Consequently $\Theta_Q$ stays off the diagonal after leaving $(s,s)$, and summing the displayed divergences through scale $J$ gives
\[
\mathbbm{1}_{(s,s)}-\lambda_J^0\otimes\lambda_J^1.
\]

We have
\[
I_\varphi(\mathbbm{1}_s)=\de(s)^{-(1-\varphi)},
\]
and $P_{i-1}^{-1}=O(P_i^{-1})$ for $i\geq1$.  Hence
\[
\cE_{\cG}(f_i^b)^{1/2} \leq\cE_{\cG}(g_i^b)^{1/2}+\cE_{\cG}(g_{i-1}^b)^{1/2} =O(P_i^{-1/2}).
\]

Apply \eqref{eq:weighted-lift} to each lift and then use the triangle inequality for $\cE_\times^{1/2}$.  The definitions of $P_i,V_i$ give
\[
\cE_\times(\Theta_Q)^{1/2} =O\left(M_{1-\varphi}\left\{\frac1{\de(s)^{1-\varphi/2}} +\frac1{\de(s)}\sum_{i\geq0}e^{-(1-\alpha)u_i/2}\right\}\right) =O\left(\frac{M_{1-\varphi}}{\de(s)^{1-\varphi/2}}\right),
\]
where the geometric sum is bounded and the last inequality uses $\varphi\geq0$.  Consequently
\begin{equation}
\cE_\times(\Theta_Q) =O\left(\frac{M_{1-\varphi}^2}{\de(s)^{2-\varphi}}\right).
\label{eq:alternating-product-energy}
\end{equation}

At the final scale $u_J=R-h(s)+O(1)$ and $V_J\asymp n$.  Moreover,
\begin{equation}
I_\varphi(\lambda_J^0),I_\varphi(\lambda_J^1)=O(n^{-1}),\qquad P_J=\de(s)^{\chi+o(1)}.                                      \label{eq:weighted-final-scale}
\end{equation}

Lemma~\ref{lem:uniform-endpoint-flow} gives a point-to-point flow $h_{s,w}$, available because of the fixed separation in
\eqref{eq:fixed-separated-source}, with
\[
\operatorname{div}h_{s,w}=\mathbbm{1}_s-\mathbbm{1}_w, \qquad \cE_{\cG}(h_{s,w}) =O\left(\de(s)^{-\chi+o(1)}+\de(w)^{-\chi+o(1)}\right).
\]

All its nonendpoint vertices have height at least a positive constant times $R$, and hence are disjoint from $Q_J^1$, which lies in the fixed outer height band of Lemma~\ref{lem:separated-multiscale-flows}; the fixed angular separation also gives $w\notin Q_J^1$.  The support separation in that lemma makes the support of $g_J^0$ disjoint from $Q_J^1$ as well.  Define the complete product flow by
\begin{equation}
\Theta:=\Theta_Q+ \mathsf L_1(-g_J^0+h_{s,w},\lambda_J^1).                    \label{eq:complete-product-flow}
\end{equation}

Indeed,
\[
\operatorname{div}_{\cG}(-g_J^0+h_{s,w}) =-(\mathbbm{1}_s-\lambda_J^0) +(\mathbbm{1}_s-\mathbbm{1}_w) =\lambda_J^0-\mathbbm{1}_w,
\]
and hence
\[
\operatorname{div}_\times\Theta =\mathbbm{1}_{(s,s)}-\lambda_J^0\otimes\lambda_J^1 +(\lambda_J^0-\mathbbm{1}_w)\otimes\lambda_J^1 =\mathbbm{1}_{(s,s)}-\mathbbm{1}_w\otimes\lambda_J^1.
\]

Thus \eqref{eq:complete-product-flow} is a unit flow to $\{w\}\times Q_J^1\subset B_w$, and the support separation just noted shows that its added lift avoids the diagonal.  Moreover,
\[
\cE_{\cG}(-g_J^0+h_{s,w}) =O\left(\de(s)^{-\chi+o(1)} +\de(w)^{-\chi+o(1)}\right),
\]
so \eqref{eq:weighted-lift} and \eqref{eq:weighted-final-scale} give
\begin{align*}
\cE_\times(\Theta-\Theta_Q) &=M_{1-\varphi}^2I_\varphi(\lambda_J^1)
   \cE_{\cG}(-g_J^0+h_{s,w})\\
&=O\left(\frac{M_{1-\varphi}^2}{n} \left(\de(s)^{-\chi+o(1)}+\de(w)^{-\chi+o(1)}\right)\right).
\end{align*}

The source term is $O(M_{1-\varphi}^2/\de(s)^{2-\varphi})$ for $\varphi\geq0$. The target term has this order when
\[
1+\beta \cdot \chi>\frac{2-\varphi}{2\alpha},
\]
which is exactly $\beta>\frac{2(1-\alpha)-\varphi}{8\alpha^2(1-\alpha)}$.  Under this condition, \eqref{eq:alternating-product-energy} and the triangle inequality give
\[
\cE_\times(\Theta)^{1/2} \leq\cE_\times(\Theta_Q)^{1/2} +\cE_\times(\Theta-\Theta_Q)^{1/2} =O\left(\frac{M_{1-\varphi}} {\de(s)^{1-\varphi/2}}\right),
\]
and therefore
\[
\cE_\times(\Theta) =O\left(\frac{M_{1-\varphi}^2}{\de(s)^{2-\varphi}}\right).
\]

Thomson's principle, \eqref{eq:shorted-diagonal-c}, and
\eqref{eq:shorted-conductance-identity} therefore give, for some
$c_\varepsilon>0$,
\begin{equation}
\p_\rho(T_w^\times<T_{\rm meet}\mid\cG) \geq c_\varepsilon a_s.                                      \label{eq:rho-escape-to-w}
\end{equation}

Together with \eqref{eq:fixed-separated-source}, this gives the lower bound in \eqref{eq:target-before-remeeting} outside a set of graphs of probability at most $\varepsilon+o(1)$.  The upper bound is immediate, so the proposition follows.
\end{proof}

For a fixed graph, recall that
\begin{equation}
R_\varphi(v):= \inf\left\{\cE_{\cG}(f): f\text{ has divergence }\mathbbm{1}_v-\pi\right\}.
\label{eq:stationary-flow-energy}
\end{equation}

\begin{theorem}[Mixing before remeeting]                       \label{thm:escape-hrg}
Under \eqref{eq:beta-window}, let $\kappa$ satisfy
\begin{equation}
2\alpha-1<\kappa<1-\chi\beta,                                \label{eq:kappa-window}
\end{equation}
Then
\begin{equation}
\p_\rho\left(T_{\rm mix}<n^\kappa<T_{\rm meet}\,\middle|\,\cG\right) \pasymp1.                                                     \label{eq:escape-positive}
\end{equation}
\end{theorem}

\begin{proof}
On the high-probability event in which a valid target candidate exists, the selected target lies within $\eta$ of angle $0$, and hence
\[
\sum_{v:\,\operatorname{dist}_{\mathbb S^1} (\vartheta(v),\vartheta(w))\leq3\eta}a_v \leq \sum_{v:\,\operatorname{dist}_{\mathbb S^1} (\vartheta(v),0)\leq4\eta}a_v.
\]

By rotational invariance, the expectation of the last sum is $O(\eta)$.  Markov's inequality therefore gives, with high probability,
\begin{equation}
\sum_{v:\,\operatorname{dist}_{\mathbb S^1} (\vartheta(v),\vartheta(w))\leq3\eta}a_v=o(1).                \label{eq:nearby-diagonal-mass}
\end{equation}

Let $L_{\cG}$ be the positive unweighted graph Laplacian,
\[
L_{\cG}f(x):=\sum_{y\sim x}(f(x)-f(y)).
\]

For an oriented edge, write $\nabla f(x,y):=f(x)-f(y)$, so that $\operatorname{div}_{\cG}\nabla f=L_{\cG}f$.  Put $H_v(x):=\e_xT_v$.  For $x\neq v$, the hitting equation reads $-1=\sum_yQ_{xy}(H_v(y)-H_v(x))$; hence \eqref{eq:voter-kernel} gives $L_{\cG}H_v(x)=\de(x)^{1-\varphi}$.  Since $\sum_xL_{\cG}H_v(x)=0$, its value at $v$ is $\de(v)^{1-\varphi}-M_{1-\varphi}$.  Thus, for every $x$,
\[
L_{\cG}H_v(x)=\de(x)^{1-\varphi} -M_{1-\varphi}\mathbbm{1}_{\{x=v\}} =-M_{1-\varphi}\bigl(\mathbbm{1}_{\{x=v\}}-\pi(x)\bigr).
\]

Thus $-M_{1-\varphi}^{-1}\nabla H_v$ has divergence $\mathbbm{1}_v-\pi$.  By Thomson's principle it is the minimum-energy flow with that divergence, and summation by parts, using $H_v(v)=0$, gives $\e_\pi T_v=M_{1-\varphi}R_\varphi(v)$.  Writing $p_t(x,y)$ for the transition probabilities, the reversible spectral representation and Green identity \cite[Lemma~2.11 and (3.40)--(3.41)]{aldous-fill-2014} now give
\begin{equation}
\int_0^t p_u(v,v)\de u \leq t\pi(v)+\pi(v)\e_\pi T_v =\frac{t\de(v)^{1-\varphi}}{M_{1-\varphi}} +\de(v)^{1-\varphi}R_\varphi(v).                           \label{eq:weighted-green}
\end{equation}

We next bound $R_\varphi(v)$ uniformly over $v\in\mathcal W_n$.  Apply Lemma~\ref{lem:separated-multiscale-flows} at its final scale $J=J(v)$, and take $g_J^0$.  Uniformly, this flow has divergence $\mathbbm{1}_v-\lambda_J^0$, energy $\de(v)^{-\chi}(\log n)^{O(1)}$, and $I_\varphi(\lambda_J^0)=O(n^{-1})$.  It is not yet admissible in the infimum in \eqref{eq:stationary-flow-energy}, because its sink is $\lambda_J^0$ rather than $\pi$.  Since $\lambda_J^0-\pi$ has total mass zero, let $\psi$ be the $\pi$-mean-zero solution of
\[
L_{\cG}\psi=\lambda_J^0-\pi.
\]

Then $\operatorname{div}_{\cG}\nabla\psi=\lambda_J^0-\pi$, so $g_J^0+\nabla\psi$ has divergence $\mathbbm{1}_v-\pi$.  Moreover, weighted Cauchy--Schwarz and the variational definition of $t_{\rm rel}(\varphi)$ \cite[(3.74) and Theorem~3.25]{aldous-fill-2014} give
\begin{align*}
\cE_{\cG}(\nabla\psi)
 &=\sum_x(\lambda_J^0(x)-\pi(x))\psi(x)\\
&\leq\left(\sum_x \frac{(\lambda_J^0(x)-\pi(x))^2}{\de(x)^{1-\varphi}}\right)^{1/2}
 \left(M_{1-\varphi}\var_\pi(\psi)\right)^{1/2}\\
&\leq\left(t_{\rm rel}(\varphi)\cE_{\cG}(\nabla\psi) \sum_x\frac{(\lambda_J^0(x)-\pi(x))^2} {\de(x)^{1-\varphi}}\right)^{1/2}.
\end{align*}

Consequently
\[
\cE_{\cG}(\nabla\psi)\leq t_{\rm rel}(\varphi)\sum_x \frac{(\lambda_J^0(x)-\pi(x))^2}{\de(x)^{1-\varphi}}.
\]

Now $M_{1-\varphi}\whpasymp n$ by Proposition~\ref{prop:giant-degree-moments}, and, with high probability, uniformly over this window,
\[
\sum_x\frac{(\lambda_J^0(x)-\pi(x))^2}{\de(x)^{1-\varphi}} \leq 2I_\varphi(\lambda_J^0)+\frac2{M_{1-\varphi}}=O(n^{-1}).
\]

The comparison preceding Corollary~\ref{cor_comparison} gives, with high probability, $t_{\rm rel}(\varphi)=\widetilde O(n^{2\alpha-1})$ for $\varphi\geq0$.  Thus the continuation has energy $\widetilde O(n^{2\alpha-2})$, which is $\de(v)^{-\chi}(\log n)^{O(1)}$ uniformly on $\mathcal W_n$ by
\eqref{eq:extreme-source-mass}.  The triangle inequality for
$\cE_{\cG}^{1/2}$ applied to $g_J^0+\nabla\psi$ therefore gives
\begin{equation}
R_\varphi(v) \leq \de(v)^{-\chi}(\log n)^{O(1)}, \qquad v\in\mathcal W_n.                                          \label{eq:green-extreme-window}
\end{equation}

The cut included in the selection of $w$ lies in an angular aperture $n^{2\alpha\beta-1}(\log n)^{O(1)}=o(\eta)$.  Hence, simultaneously for every $v$ at angular distance at least $2\eta$ from $w$,
\cite[Lemma~67]{MR4816414} and the one-cut bound from the Dirichlet
principle \cite[Proposition~3.38]{aldous-fill-2014} give
\[
\cR_{\cG}(v\leftrightarrow w) \geq \de(w)^{-\chi}(\log n\log\log n)^{-3+o(1)}.
\]

The embedded jump chain is simple random walk, so an excursion from $v$ hits $w$ before returning to $v$ with probability $1/(\de(v)\cR_{\cG}(v\leftrightarrow w))$.  Since departures from $v$ occur at rate $\de(v)^\varphi$, \eqref{eq:weighted-green} gives
\[
\e_v\bigl[\text{departures from $v$ by time $t$}\bigr] \leq \de(v)R_\varphi(v)+\frac{t\de(v)}{M_{1-\varphi}}.
\]

A union bound over excursions, followed by the strong Markov property for the walker started from a uniform neighbour $U_v$ of $v$, yields, with high probability and uniformly for $v\in\mathcal W_n$ outside the $3\eta$-neighbourhood of $w$,
\begin{equation}
\p_{v,U_v}(T_w^\times\leq t) \leq (\log n)^{O(1)}\left\{\left(\frac{\de(w)}{\de(v)}\right)^\chi +\frac{t\de(w)^\chi}{n}\right\}+o(1).                              \label{eq:early-w-varphi}
\end{equation}

Since $T_w^\times$ is invariant under exchanging the coordinates, we average this estimate over $v$ with weights $a_v$ and over its uniform neighbour $U_v$.  Sources within angular distance $3\eta$ of $w$ have total $a$-mass $o(1)$ with high probability by
\eqref{eq:nearby-diagonal-mass}; sources outside $\mathcal W_n$ have total
$a$-mass $o(1)$ with high probability by
\eqref{eq:extreme-source-mass}.  On the remaining sources,
$\de(v)=n^{1/(2\alpha)}(\log n)^{O(1)}$ uniformly.  Taking $t=n^\kappa$ in \eqref{eq:early-w-varphi} and using \eqref{eq:target-degree} therefore gives, with high probability,
\begin{equation}
\p_\rho(T_w^\times\leq n^\kappa\mid\cG) \leq o(1) +n^{\chi(\beta-1/(2\alpha))+o(1)} +n^{\kappa+\beta\chi-1+o(1)} =o(1),                                                        \label{eq:early-w-zero}
\end{equation}
by \eqref{eq:beta-window} and \eqref{eq:kappa-window}.  This uses the stationary-resistance estimate only on the extreme window to which Lemma~\ref{lem:separated-multiscale-flows} applies.

Proposition~\ref{prop:escape-to-target} and
\eqref{eq:early-w-zero} concern the same target $w$, fixed before the
exit state was sampled.  Thus, for every $\varepsilon>0$, there is $c_\varepsilon>0$ such that
\[
\p_\rho(n^\kappa<T_w^\times<T_{\rm meet}\mid\cG) \geq \p_\rho(T_w^\times<T_{\rm meet}\mid\cG) -\p_\rho(T_w^\times\leq n^\kappa\mid\cG) \geq c_\varepsilon
\]
outside a set of graphs of probability at most $\varepsilon+o(1)$.

For an initial law $\lambda$ on $V^2$, write
\[
s_\lambda(t):=1-\min_{z\in V^2} \frac{\p_\lambda(Z_t=z)}{(\pi\otimes\pi)(z)}
\]
for the product-chain separation distance.

The relaxation-time comparison and $\kappa>2\alpha-1$ imply, with high probability, the uniform separation bound
\[
s_\rho(n^\kappa)=o(1).
\]

For the optimal strong stationary time in the definition of $T_{\rm mix}$, with high probability,
\[
\p_\rho(T_{\rm mix}\geq n^\kappa\mid\cG) =s_\rho(n^\kappa)=o(1).
\]

Combining the last two estimates, we conclude that
\[
\p_\rho(T_{\rm mix}<n^\kappa<T_{\rm meet}\mid\cG) \geq c_\varepsilon/2
\]
outside a set of graphs of probability at most $\varepsilon+o(1)$, as claimed.
\end{proof}

\subsection{Critical escape}

In this section we consider $\varphi=2-2\alpha$. Recall that $h(v):=R-r_v$ is the distance from the boundary, and $\chi=4\alpha(1-\alpha)$. Recall $B_w$ from \eqref{eq:product-target-set} and $\eta=(\log n)^{-3}$, and let $T_w^\times$ again denote the hitting time of $B_w$.

At $\varphi=2-2\alpha$ the probability to meet at a vertex $v$ is proportional to $\de(v)^{2\alpha}$, which is the critical parameter for the empirical degree moment of the giant.  Therefore the sum of $\de(v)^{2\alpha}$ is no longer dominated by the largest $O(1)$ terms, and we must instead combine a polynomial number of possible meeting positions to collect a positive probability event.

Fix a closed angular arc $\mathcal I$ of positive length at fixed positive distance from angle $0$.  For $0<\delta<1-\alpha$, define
\[
k_n:=\lfloor n^\delta\rfloor, \qquad j_n:=\lceil(\log n)^2\rceil, \qquad H_n:=\sum_{j=j_n}^{k_n}\frac1j\asymp\log n,
\]
and list the vertices in $\mathcal I$ in decreasing degree order as $v_1,v_2,\ldots$.

\begin{lemma}[Polynomial source family]                         \label{lem:polynomial-sources}
Take $\epsilon$ in Lemma~\ref{lem:separated-multiscale-flows} smaller than the fixed angular separation between $\mathcal I$ and the target arc.  With high probability, its conclusions then hold simultaneously for every $v_j$, $j_n\leq j\leq k_n$.  If $J=J(j)$ is the final scale for $v_j$, then
\begin{align}
 K_J(v_j)&\asymp n^{(1-\alpha)/\alpha}j^{-1/\alpha},             \label{eq:critical-initial-count}\\
P_{J,j}&:=\de(v_j)e^{-(2\alpha-1)(R-h(v_j))/2} \asymp\frac{n^{2-2\alpha}}j.                          \label{eq:critical-final-scale}
\end{align}
\end{lemma}

\begin{proof}
By Lemma~\ref{lem:sector-critical-orders}, for some fixed $A<\infty$, with high probability every relevant $v_j$ has height at least
\[
\frac1\alpha\log\frac n{k_n}-A \geq\left(\frac12+\epsilon_0\right)R
\]
for some fixed $\epsilon_0>0$, where the inequality follows from $\delta<1-\alpha$.  Thus the height assumption in Lemma~\ref{lem:separated-multiscale-flows} holds uniformly.  At the final scale $u_J=R-h(v_j)+O(1)$, the appendix estimates,
\eqref{eq:initial-half-tile-count} and the definition of $P_{J,j}$ give
\eqref{eq:critical-initial-count} and \eqref{eq:critical-final-scale}.

For a prescribed point $x$ above the height cutoff in the last display,
\eqref{eq:initial-half-tile-count} shows that the candidate count at
every scale is at least a fixed positive multiple of $n^{(1-\alpha-\delta)/\alpha}$.  Consequently the Palm failure estimate in Lemma~\ref{lem:separated-multiscale-flows} is at most $\exp\{-c n^{(1-\alpha-\delta)/\alpha}\}$, uniformly over the scale and $b\in\{0,1\}$, for some fixed $c>0$.

Write $\p_x$ for the Palm law obtained by adding $x$.  Mecke's formula gives the following bound on the expected number of failures above the cutoff, summed over points, scales and both values of $b$:
\begin{align*}
&\e\left[
 \sum_{\substack{x\in\mathcal P_n\\
h(x)\geq\alpha^{-1}\log(n/k_n)-A}} \sum_{i=0}^{J(x)}\sum_{b=0}^1 \mathbbm1_{\{(x,i,b)\text{ fails}\}}
 \right] \\
&\quad= \int_{\{h(x)\geq\alpha^{-1}\log(n/k_n)-A\}} \sum_{i=0}^{J(x)}\sum_{b=0}^1
 \p_x\bigl((x,i,b)\text{ fails}\bigr)\,\lambda_n(\de x) \\
&\quad\leq 2n\left(\frac R\ell+1\right) \exp\{-c n^{(1-\alpha-\delta)/\alpha}\} \longrightarrow0.
\end{align*}
Here the factor $2$ accounts for the two values of $b$, $J(x)+1\leq R/\ell+1$, and $\lambda_n(B_O(R))=n$.  Together with the initial high-probability height event, Markov's inequality now gives simultaneous success for all the required vertices with high probability.
\end{proof}

\begin{proposition}[Critical aggregate conductance]
\label{prop:critical-aggregate-conductance}
For fixed $0<\beta<1/(2\alpha)$, choose $w$ by the same function of the point process as before Proposition~\ref{prop:escape-to-target}.  The argument proving
\eqref{eq:target-degree} uses only this upper bound on $\beta$, and hence
$\de(w)=n^{\beta+o(1)}$ with high probability.

For every fixed $0<\delta<1-\alpha$ there are constants $\beta_0,c_1>0$ such that, for every fixed $0<\beta<\beta_0$, with high probability,
\begin{equation}
\p_\rho(T_w^\times<T_{\rm meet}\mid\cG)\geq c_1 .              \label{eq:critical-conductance-escape}
\end{equation}
\end{proposition}

\begin{proof}
Take $\beta_0:=1/(2\alpha)$. For the objects associated with $v_j$, write $g_{i,j}^b$ and $\lambda_{i,j}^b$ for $g_i^b$ and $\lambda_i^b$ in Lemma~\ref{lem:separated-multiscale-flows}, where $J=J(j)$ is the final scale.  For each $j_n\leq j\leq k_n$, put $\lambda_{-1,j}^0=\lambda_{-1,j}^1=\mathbbm{1}_{v_j}$ and $g_{-1,j}^0=g_{-1,j}^1=0$.  With the endpoint flow from Lemma~\ref{lem:uniform-endpoint-flow}, define
\begin{align*}
\Theta_j:={}&\sum_{i=0}^J\left\{ \mathsf L_1(g_{i,j}^0-g_{i-1,j}^0,\lambda_{i-1,j}^1)
 +\mathsf L_2(\lambda_{i,j}^0,g_{i,j}^1-g_{i-1,j}^1)\right\}\\
&\quad+\mathsf L_1(-g_{J,j}^0+h_{v_j,w},\lambda_{J,j}^1).
\end{align*}

The moving-coordinate flow in the last lift has divergence $\lambda_{J,j}^0-\mathbbm{1}_w$.  Thus, as in
\eqref{eq:alternating-product-flow}, $\Theta_j$ is a unit flow from
$(v_j,v_j)$ to $\{w\}\times\operatorname{supp}\lambda_{J,j}^1\subset B_w$. The support conclusions of Lemmas~\ref{lem:separated-multiscale-flows} and~\ref{lem:uniform-endpoint-flow} imply that $\operatorname{supp}\lambda_{J,j}^1$ is disjoint from the support of $-g_{J,j}^0+h_{v_j,w}$, so the final lift also avoids the diagonal. Write
\[
A_j:=\mathsf L_1(g_{0,j}^0,\mathbbm{1}_{v_j}).
\]

Intersect the high-probability events in Lemmas
\ref{lem:polynomial-sources} and~\ref{lem:sector-critical-orders}, and in
Proposition~\ref{prop:giant-degree-moments}.  We may choose a fixed $C\geq1$, independent of $n,j,i,b$ and $\beta$, such that throughout the stated ranges
\begin{align}
\cE_{\cG}(g_{i,j}^b) &\leq\frac{C}{\de(v_j)e^{-(2\alpha-1)u_i/2}}, &I_{2-2\alpha}(\lambda_{i,j}^b)
 &\leq\frac{C}{\de(v_j)e^{u_i/2}},                    \label{eq:critical-fixed-flow-bounds}\\
\frac nC\leq M_{2\alpha-1}&\leq Cn,
&M_{2\alpha}&\leq Cn\log n,\notag\\
\de(v_j)^{2\alpha}&\geq\frac n{Cj},
&P_{J,j}&\geq\frac{n^{2-2\alpha}}{Cj},\notag\\
\de(v_j)e^{u_J/2}&\geq\frac nC.                       \label{eq:critical-fixed-order-bounds}
\end{align}

This constant depends only on the fixed model and construction parameters (including $\delta$).  For each fixed $\beta$, further intersect with the event in Lemma~\ref{lem:uniform-endpoint-flow}. All implicit constants below are independent of $n,j$ and $\beta$; the rate in its $o(1)$ term, and hence how large $n$ must be, may depend on the fixed $\beta$.

Thus $A_j$ is the $i=0$ first-coordinate summand in
\eqref{eq:alternating-product-flow}.  By
\eqref{eq:weighted-lift}, \eqref{eq:critical-fixed-flow-bounds} and
\eqref{eq:critical-fixed-order-bounds},
\[
\cE_\times(A_j) \leq \frac{CM_{2\alpha-1}^2}{\de(v_j)^{2\alpha}} \leq C^4nj.
\]
The product-edge supports of the $A_j$ are pairwise disjoint: in $A_j$ the second coordinate is frozen at the distinct vertex $v_j$.

For the remaining alternating lifts, the energy triangle inequality for $g_{i,j}^b-g_{i-1,j}^b$, \eqref{eq:weighted-lift} and
\eqref{eq:critical-fixed-flow-bounds} give
\begin{align*}
&\sum_{i=1}^J \cE_\times\!\left(\mathsf L_1 (g_{i,j}^0-g_{i-1,j}^0,\lambda_{i-1,j}^1)\right)^{1/2} +\sum_{i=0}^J \cE_\times\!\left(\mathsf L_2
 (\lambda_{i,j}^0,g_{i,j}^1-g_{i-1,j}^1)\right)^{1/2}\\
&\quad\leq 2C M_{2\alpha-1}\left\{ \sum_{i=1}^J \left(\de(v_j)e^{-(2\alpha-1)u_i/2}\right)^{-1/2} \left(\de(v_j)e^{u_{i-1}/2}\right)^{-1/2}
 \right.\\
&\hspace{42mm}\left. +\sum_{i=0}^J \left(\de(v_j)e^{-(2\alpha-1)u_i/2}\right)^{-1/2}
 \left(\de(v_j)e^{u_i/2}\right)^{-1/2}\right\}\\
&\quad\leq\frac{2C(1+e^{\ell/4})M_{2\alpha-1}}{\de(v_j)} \sum_{i=0}^J e^{-(1-\alpha)u_i/2} =O\left(\frac{M_{2\alpha-1}}{\de(v_j)}\right),
\end{align*}
where we used the geometric sum and the fact that $u_i-u_{i-1}=\ell$ is fixed.  The omitted first-coordinate lift at $i=0$ is precisely $A_j$.

At the final scale, \eqref{eq:critical-fixed-flow-bounds} and
\eqref{eq:critical-fixed-order-bounds} give
\[
I_{2-2\alpha}(\lambda_{J,j}^1) \leq\frac{C}{\de(v_j)e^{u_J/2}} \leq\frac{C^2}{n}.
\]
Since $0\leq R-h(v_j)-u_J<\ell$, the first bound in
\eqref{eq:critical-fixed-flow-bounds} gives
$\cE_{\cG}(g_{J,j}^0)\leq C/P_{J,j}$.  Hence the energy triangle inequality and \eqref{eq:uniform-endpoint-flow} give
\[
\cE_{\cG}(-g_{J,j}^0+h_{v_j,w})^{1/2} \leq C^{1/2}P_{J,j}^{-1/2} +n^{o(1)}\left(\de(v_j)^{-\chi/2}+\de(w)^{-\chi/2}\right).
\]
Applying \eqref{eq:weighted-lift} to this final lift and then the energy triangle inequality to $\Theta_j-A_j$ therefore gives
\[
\cE_\times(\Theta_j-A_j)^{1/2} =O\left(\frac{M_{2\alpha-1}}{\de(v_j)} +\frac{M_{2\alpha-1}}{\sqrt n} \left\{P_{J,j}^{-1/2} +n^{o(1)}\left(\de(v_j)^{-\chi/2} +\de(w)^{-\chi/2}\right)\right\}\right).
\]

Squaring this bound gives
\begin{equation*}
\cE_\times(\Theta_j-A_j)=O\left( \frac{M_{2\alpha-1}^2}{\de(v_j)^2} +\frac{M_{2\alpha-1}^2}{n}\left(P_{J,j}^{-1} +n^{o(1)}\de(v_j)^{-\chi} +n^{o(1)}\de(w)^{-\chi}\right)\right).
\end{equation*}

The fixed bounds in \eqref{eq:critical-fixed-order-bounds} and $\de(w)=n^{\beta+o(1)}$ now give the following estimates.  All $o(1)$ terms are uniform in $j$ for each fixed $\beta$.
\begin{align*}
\frac{M_{2\alpha-1}^2}{\de(v_j)^2} &=O\left(n^{2-1/\alpha}j^{1/\alpha}\right), &\frac{M_{2\alpha-1}^2}{nP_{J,j}}
 &=O\left(jn^{2\alpha-1}\right),\\
\frac{M_{2\alpha-1}^2}{n}n^{o(1)}\de(v_j)^{-\chi} &=O\left(n^{2\alpha-1+o(1)}j^{2(1-\alpha)}\right), &\frac{M_{2\alpha-1}^2}{n}n^{o(1)}\de(w)^{-\chi} &=O\left(n^{1-\beta\chi+o(1)}\right).
\end{align*}

Since $2\alpha-1<2-1/\alpha$, $1<1/\alpha$ and $2(1-\alpha)<1/\alpha$, for each fixed $\beta>0$ and all sufficiently large $n$, uniformly in $j$,
\begin{equation*}
\cE_\times(\Theta_j-A_j) =O\left(n^{2-1/\alpha}j^{1/\alpha} +n^{1-\beta\chi/2}\right).
\end{equation*}

Put $r_j:=(jH_n)^{-1}$ and
\[
\Theta^{\rm crit}:=\sum_{j=j_n}^{k_n}r_j\Theta_j.
\]

This is a unit flow from a probability measure on the diagonal to a probability measure on $B_w$.  The triangle inequality gives
\begin{align}
\cE_\times(\Theta^{\rm crit})^{1/2} &\leq \left(\sum_{j=j_n}^{k_n}r_j^2\cE_\times(A_j)\right)^{1/2}
 +\sum_{j=j_n}^{k_n}r_j\cE_\times(\Theta_j-A_j)^{1/2} \notag\\
&\leq C^2\sqrt{\frac n{H_n}} +O\left(\frac{n^{1-1/(2\alpha)}k_n^{1/(2\alpha)}}{H_n} +n^{(1-\beta\chi/2)/2}\right).                  \label{eq:critical-triangle}
\end{align}

Here we used $\sum_{j=1}^{k_n}j^{-1+1/(2\alpha)}\leq 2\alpha k_n^{1/(2\alpha)}$.  The squares of the two terms inside the $O(\cdot)$ in \eqref{eq:critical-triangle}, divided by $n/H_n$, are, respectively,
\[
O\left(\frac1{H_n} n^{-(1-\alpha-\delta)/\alpha}\right) \longrightarrow 0, \qquad O\left(H_n n^{-\beta\chi/2}\right) \longrightarrow 0.
\]

Thus, for each fixed $\beta$ and all sufficiently large $n$,
\[
\cE_\times(\Theta^{\rm crit})^{1/2} \leq2C^2\sqrt{n/H_n}.
\]

Since $H_n\sim\delta\log n$, we have $H_n\geq\frac\delta2\log n$ for all sufficiently large $n$.  Hence, with high probability,
\begin{equation}
\cE_\times(\Theta^{\rm crit}) \leq\frac{8C^4}{\delta}\frac n{\log n}.                 \label{eq:critical-flow-energy}
\end{equation}

After the diagonal is shorted, Thomson's principle and
\eqref{eq:critical-flow-energy} give
\[
\cC_\times(\Delta\leftrightarrow B_w) \geq\frac{\delta}{8C^4}\frac{\log n}{n}.
\]

Moreover, \eqref{eq:shorted-diagonal-c} and
\eqref{eq:critical-fixed-order-bounds} give
\[
\cC_\times(\Delta) =\frac{2M_{2\alpha}}{M_{2\alpha-1}^2} \leq2C^3\frac{\log n}{n}.
\]

Consequently \eqref{eq:shorted-conductance-identity} gives the claim
\[
\p_\rho(T_w^\times<T_{\rm meet}\mid\cG)\geq \frac{\delta}{16C^7},
\]
a constant independent of $n$ and $\beta$.
\end{proof}

The next strengthening of \eqref{eq:green-extreme-window} supplies the delayed-target estimate for the degree scales not used as flow sources; recall the definition \eqref{eq:stationary-flow-energy} of $R_\varphi$.

\begin{lemma}[Polynomial-degree stationary potentials]
\label{lem:critical-stationary-potential}
At $\varphi=2-2\alpha$, for every fixed $\gamma>0$, with high probability uniformly over vertices with $h(v)\geq\gamma R$,
\begin{equation}
R_{2-2\alpha}(v)\leq \de(v)^{-\chi}n^{o(1)}.                                          \label{eq:critical-stationary-potential}
\end{equation}
\end{lemma}

\begin{proof}
Choose two fixed closed angular arcs at positive distance from one another.  Poisson concentration shows that, with high probability, each arc contains polylogarithmically many vertices $z$ satisfying
\[
\frac{R}{2\alpha}-2\log\log n\leq h(z) \leq\frac{R}{2\alpha}-\log\log n.
\]

Choose one such vertex in each arc.  The Palm failure estimate in Lemma~\ref{lem:separated-multiscale-flows}, Mecke's formula and a union bound make its final-scale flow available for both selected vertices. Since $2-2\alpha<1$ and $M_{2\alpha-1}\whpasymp n$ by Proposition~\ref{prop:giant-degree-moments}, the continuation argument proving \eqref{eq:green-extreme-window} applies unchanged and gives
\[
R_{2-2\alpha}(z)\leq \de(z)^{-\chi}(\log n)^{O(1)}.
\]

For every $v$ with $h(v)\geq\gamma R$, one of the selected vertices $z$ is at fixed positive angular separation from $v$. Lemma~\ref{lem:uniform-endpoint-flow} gives the required flow $h_{v,z}$ simultaneously for all such $v$.  Concatenating it with the flow from $z$ to $\pi$ and using the energy triangle inequality proves
\eqref{eq:critical-stationary-potential}; the selected $z$ has degree within a
polylogarithmic factor of the maximum, so its term is absorbed by the $n^{o(1)}$ factor.
\end{proof}

\begin{theorem}[Critical mixing before remeeting]
\label{thm:critical-escape}
At $\varphi=2-2\alpha$, there is a fixed $c>0$ such that, with high probability,
\begin{equation}
\p_\rho(T_{\rm mix}<n^\alpha<T_{\rm meet}\mid\cG)\geq c.          \label{eq:critical-escape}
\end{equation}
\end{theorem}

\begin{proof}
Apply Proposition~\ref{prop:critical-aggregate-conductance} with $\delta=(1-\alpha)/2$, and let $\beta_0,c_1$ be its constants.  The normalised truncated moment estimate in Lemma~\ref{lem:truncated-critical-moment} allows us to choose a fixed $\gamma>0$ so small that
\begin{equation}
\frac1{M_{2\alpha}} \sum_{v:h(v)<\gamma R}\de(v)^{2\alpha}\leq\frac{c_1}{4}          \label{eq:critical-low-source-mass}
\end{equation}
with high probability.  Next choose
\[
0<\beta<\min\left\{\gamma,\frac{1-\alpha}{\chi}, \frac1{2\alpha},\beta_0\right\}
\]
and use the target of Proposition~\ref{prop:critical-aggregate-conductance}.

At criticality, put $a_v:=\de(v)^{2\alpha}/M_{2\alpha}$.  The sector edge-cut estimate used in
\eqref{eq:early-w-varphi} applies simultaneously outside the
$3\eta$-aperture about $w$.  The rotational-invariance argument proving
\eqref{eq:nearby-diagonal-mass}
gives, with high probability,
\[
\sum_{v:\,\operatorname{dist}_{\mathbb S^1} (\vartheta(v),\vartheta(w))\leq3\eta}a_v=o(1).
\]

For each $v$, let $U_v$ be a uniform neighbour of $v$.  For $h(v)\geq\gamma R$, Lemma~\ref{lem:critical-stationary-potential} and the excursion calculation leading to \eqref{eq:early-w-varphi} give, uniformly outside that aperture (using $\de(v)\geq n^{\gamma+o(1)}$ from \eqref{eq:uniform-degree-height}),
\[
\p_{v,U_v}(T_w^\times\leq n^\alpha) \leq n^{-\chi(\gamma-\beta)+o(1)} +n^{\alpha+\beta\chi-1+o(1)}=o(1).
\]

Averaging with the full critical source law and using
\eqref{eq:critical-low-source-mass} therefore yields, with high probability,
\[
\p_\rho(T_w^\times\leq n^\alpha\mid\cG)\leq c_1/4+o(1).
\]

Together with \eqref{eq:critical-conductance-escape}, this gives, with high probability,
\[
\p_\rho(n^\alpha<T_w^\times<T_{\rm meet}\mid\cG)\geq c_1/2.
\]

Finally Corollary~\ref{cor_comparison} implies $s_\rho(n^\alpha)=o(1)$ with high probability.  The strong-stationary-time argument at the end of the proof of Theorem~\ref{thm:escape-hrg} now proves \eqref{eq:critical-escape}.
\end{proof}

 \section*{Acknowledgements}

JF is grateful for the hospitality of Aarhus University, where the majority of this work was done. This article is based upon work from COST Action 24122 mSPACE, supported by COST (European Cooperation in Science and Technology), \url{www.cost.eu}.

The authors used a large language model during the preparation of this manuscript for editorial and expository support. The ideas of the proofs are due to the authors and all final wording was checked and approved by the authors, who take full responsibility for the content of the manuscript.

\printbibliography

\appendix

\section{Degree moments on the giant component}

This appendix proves the degree estimates used throughout the main argument. The central result determines the order of every power sum of the degrees on the giant component. Its proof first treats the critical moment and then derives the subcritical and supercritical cases. The remaining results identify the consequences needed in the main text: the stationary collision mass, concentration of a supercritical source law near the largest degrees, degree order within a fixed sector, a uniform comparison of degree and height, and a bound on the critical moment below a fixed height.

Recall that $M_q=\sum_{v\in V(\cG)}\de(v)^q$ for $q\in\mathbb R$. When $q>0$, we also write $M_q^{\rm full}:=\sum_v\de(v)^q$. We repeatedly use that the estimates stated with extremely high probability in \cite{dieter_mixing} may be intersected over polynomially many vertices or layers. In particular, the degree comparison in \cite[Proposition~13]{dieter_mixing} may be used uniformly over all vertices in any deterministic height range.

The following proposition supplies the moment estimates from which all later degree calculations follow.

\begin{proposition}[Giant-component degree moments]
\label{prop:giant-degree-moments}
For every fixed $\alpha\in(1/2,1)$, $\nu>0$, and $q\in\mathbb R$, when $q>2\alpha$
\[
M_q \pasymp n^{q/(2\alpha)}
\]
and otherwise
\begin{equation}
M_q \whpasymp
 \begin{cases}
  n,&q<2\alpha,\\
n\log n,&q=2\alpha.
 \end{cases}                                             \label{eq:all-degree-moments}
\end{equation}
\end{proposition}

\begin{proof}
We first consider $q\leq0$. By \cite[Theorem~12 and Lemma~14]{MR4816414}, and also by \cite[Theorem~16 and Corollary~17]{dieter_mixing}, there are fixed $0<c<C<\infty$ such that, with high probability,
\[
cn\leq |V(\cG)|\leq Cn,\qquad |E(\cG)|\leq Cn.
\]

On this event, at least $cn/2$ vertices have degree at most $4C/c$. Since every degree in the giant is at least one, $M_q\leq |V(\cG)|$ and these bounded-degree vertices contribute at least a fixed positive constant each to $M_q$. Hence, $M_q\whpasymp n$.

We next prove the critical estimate. Put
\[
\omega_n=\log\log R,\qquad \ell_n=\left\lceil\max\left\{2\log R+\omega_n, \frac{2\log R}{1-\alpha}\right\}\right\rceil,
\]
and
\[
u_n=\left\lfloor\frac{R}{2\alpha} -\frac{\log R}{\alpha}-\omega_n\right\rfloor.
\]
We work in height coordinates. By \cite[Proposition~14(i)]{dieter_mixing}, applied to the corresponding radial layers, and the Poisson concentration estimate in \cite[Lemma~12]{dieter_mixing}, with high probability every unit-height layer
\[
\{v:m\leq h(v)<m+1\},\qquad \ell_n\leq m\leq u_n,
\]
contains $\asymp ne^{-\alpha m}$ points. The lower bound in the definition of $\ell_n$ places these layers inside the giant core by \cite[Remark~20 and Theorem~16]{dieter_mixing}. Moreover,
\cite[Proposition~13]{dieter_mixing}, followed by a union bound over the
$O(R)$ layers and their vertices, gives
\[
\de(v)\asymp e^{m/2}
\]
for every vertex in such a layer. Consequently, each layer contributes $\asymp n$ to $M_{2\alpha}$. Since $u_n-\ell_n\asymp R$, there is a fixed $c_0>0$ such that, with high probability,
\begin{equation}
M_{2\alpha}\geq c_0nR                                   \label{eq:critical-lower}
\end{equation}

For the upper bound, there is a fixed $C_0<\infty$ such that, under the Palm law at a point of height $h$, its full degree is Poisson with mean at most $C_0e^{h/2}$. For every fixed $r>0$, a Poisson variable $X$ with mean $\lambda$ satisfies $\e X^r=O_r(1+\lambda^r)$. Palm--Mecke and the height intensity therefore give, for every interval $A\subseteq[0,R]$,
\begin{equation}
\e\left[\sum_{v:h(v)\in A}\de(v)^{2\alpha}\right] =O\left(n\int_A e^{-\alpha h}(1+e^{\alpha h})\,\de h\right) =O\bigl(n(1+|A|)\bigr).                                 \label{eq:critical-palm-bound}
\end{equation}

Let $H_n=R/(2\alpha)+\log R$. The intervals $[0,\ell_n)$ and $(u_n,H_n]$ have total length $O(\log R)$. Hence, \eqref{eq:critical-palm-bound} and Markov's inequality show that their combined contribution is $o(nR)$ with high probability. On $[\ell_n,u_n]$, the preceding layer bounds give an upper contribution $O(nR)$. Finally,
\[
\p(\exists v:h(v)>H_n)=O(ne^{-\alpha H_n}) =O_\nu(R^{-\alpha})=o(1).
\]

Since $R=2\log(n/\nu)=2\log n+O_\nu(1)$, \eqref{eq:critical-lower} yields fixed constants $0<c<C<\infty$ such that
\begin{equation}
\p\left(cn\log n\leq M_{2\alpha}\leq M_{2\alpha}^{\rm full} \leq Cn\log n\right)\longrightarrow1.                   \label{eq:critical-degree-moment}
\end{equation}

We next record a degree-tail estimate used below. Write $N_{\geq k}^{\rm full}=|\{v:\de(v)\geq k\}|$. Under the Palm law at a point of height $h$, its degree is Poisson with mean at most $C_0e^{h/2}$. Hence, Palm--Mecke and the change of variables $s=C_0e^{h/2}$ give
\begin{equation}
\e N_{\geq k}^{\rm full} =O\left(n\int_0^R e^{-\alpha h} \p\bigl(\operatorname{Poi}(C_0e^{h/2})\geq k\bigr)\,\de h\right) =O(nk^{-2\alpha}),\qquad k\geq1.                          \label{eq:full-degree-tail}
\end{equation}

Indeed, when $s\leq k/2$, a Chernoff bound gives $\p(\operatorname{Poi}(s)\geq k)\leq e^{-\Omega(k)}$, whereas for $s\geq k/2$ we use the trivial bound one and integrate $s^{-2\alpha-1}$. The first region contributes at most a constant multiple of $k^{-2\alpha}$ as well, with bounded $k$ absorbed by the $O(\cdot)$ notation.

Now let $0<q<2\alpha$. If $q\leq1$, then
\[
|V(\cG)|\leq M_q\leq M_1=2|E(\cG)|,
\]
so Theorem~16 and Corollary~17 of \cite{dieter_mixing} prove the claim. Suppose that $1<q<2\alpha$, and let $D_k^{\rm full}=|\{v:\de(v)=k\}|$. By
\cite[Theorem~2.2]{gugelmann2012random}, there is a fixed $\eta>0$ such
that, with high probability,
\begin{equation}
D_k^{\rm full}=O\bigl(n(1+k)^{-2\alpha-1}\bigr), \qquad 1\leq k\leq n^\eta                              \label{eq:small-degree-counts}
\end{equation}
uniformly in $k$. We briefly justify the use of the fixed-size result in the present Poissonised model. Conditional on $N=m$, the graph is the fixed-size hyperbolic model with parameter
\[
C_m:=R-2\log m=-2\log\nu+2\log(n/m).
\]
On the event $|N-n|\leq n^{2/3}$, which has probability tending to one, $C_m=-2\log\nu+o(1)$. The estimates in the proof of the cited theorem are uniform for $C_m$ in a fixed compact neighbourhood of $-2\log\nu$. After shrinking $\eta$ if necessary, the estimate remains valid after conditioning and then averaging over $N$.

On the event in \eqref{eq:small-degree-counts},
\[
\sum_{v:\de(v)\leq n^\eta}\de(v)^q =O\left(n\sum_{k\geq1}k^{q-2\alpha-1}\right)=O_q(n).
\]

For the remaining degrees, summation by parts in
\eqref{eq:full-degree-tail} gives
\begin{equation}
\e\left[\sum_{v:\de(v)>n^\eta}\de(v)^q\right] =O_q\bigl(n(n^\eta)^{q-2\alpha}\bigr)=o(n).             \label{eq:subcritical-degree-tail}
\end{equation}

Markov's inequality therefore shows that the contribution above $n^\eta$ is $o(n)$ with high probability. We obtain $M_q\leq M_q^{\rm full}=O_q(n)$ with high probability, while $M_q\geq|V(\cG)|\geq cn$ with high probability.

Finally, suppose that $q>2\alpha$. Put $b_n=n^{1/(2\alpha)}$ and $\Delta_n=\max_v\de(v)$. Theorems~3.1 and~3.3 of \cite{MR5096735}, whose statements also apply to the Poissonised model, give
\begin{equation}
\Delta_n/b_n\xrightarrow{\ d\ }Z,                        \label{eq:maximum-degree-limit}
\end{equation}
where $Z$ is finite and strictly positive almost surely. They also imply that the maximum-degree vertex is the point of smallest radius with high probability. The number of points above the giant-core height $2\log R/(1-\alpha)$ is Poisson with a mean tending to infinity. Hence, with high probability there is at least one such point. The point of smallest radius then also lies in the giant core, and
\cite[Remark~20 and Theorem~16]{dieter_mixing} shows that the maximum-degree
vertex belongs to $\cG$. It follows from \eqref{eq:maximum-degree-limit} that
\[
M_q\geq\Delta_n^q\pasymp b_n^q.
\]

For the reverse bound, summation by parts in \eqref{eq:full-degree-tail} gives
\[
\e\left[\sum_{v:\de(v)\leq b_n}\de(v)^q\right] =O_q\bigl(n b_n^{q-2\alpha}\bigr)=O_q(b_n^q).
\]

Moreover, \eqref{eq:full-degree-tail} gives $\e N_{>b_n}^{\rm full}=O(1)$, so $N_{>b_n}^{\rm full}$ is bounded in probability. Equation~\eqref{eq:maximum-degree-limit} shows the same for $\Delta_n/b_n$. Thus, for every $\varepsilon>0$, there is $C_\varepsilon<\infty$ such that, outside an event of probability at most $\varepsilon+o(1)$,
\[
M_q\leq M_q^{\rm full} \leq \sum_{v:\de(v)\leq b_n}\de(v)^q +N_{>b_n}^{\rm full}\Delta_n^q \leq C_\varepsilon b_n^q.
\]

Together with the lower bound, this is $M_q\pasymp n^{q/(2\alpha)}$.
\end{proof}

The next corollary converts the moment estimates into the collision probability of the stationary distribution used in the main text.

\begin{corollary}[Stationary collision mass]
\label{cor:degree-collision}
\leavevmode For every fixed $\varphi\geq0$, define
\[
\pi_\varphi(v)=\frac{\de(v)^{1-\varphi}}{M_{1-\varphi}}.
\]

Then $\sum_v\pi_\varphi(v)^2=o(1)$ with high probability. At $\varphi=2-2\alpha$, one has
\[
\sum_v\pi_{2-2\alpha}(v)^2\whpasymp n^{-1}.
\]
\end{corollary}

\begin{proof}
We have
\[
\sum_v\pi_\varphi(v)^2 =\frac{M_{2-2\varphi}}{M_{1-\varphi}^2}.
\]

The denominator is of order $n^2$ with high probability by Proposition~\ref{prop:giant-degree-moments}. If $2-2\varphi\leq2\alpha$, the numerator is $O(n\log n)$ with high probability. If $2-2\varphi>2\alpha$, then $M_{2-2\varphi}/n^{(1-\varphi)/\alpha}$ is bounded in probability, while $(1-\varphi)/\alpha<2$. In either case, the ratio converges to zero in probability.

At $\varphi=2-2\alpha$, the numerator and denominator are $M_{4\alpha-2}$ and $M_{2\alpha-1}^2$. Since $4\alpha-2<2\alpha$ and $2\alpha-1<2\alpha$, Proposition~\ref{prop:giant-degree-moments} gives orders $n$ and $n^2$, respectively, with high probability.
\end{proof}

The first auxiliary lemma shows that a source law with exponent above the critical moment is concentrated on a polylogarithmic set near the maximal height. This localisation permits uniform estimates over all vertices carrying asymptotically relevant source mass.

\begin{lemma}[Concentration in window]
\label{lem:extreme-source-window}
Fix $0\leq\varphi<2-2\alpha$, put $p=2-\varphi$, and set
\[
\mathcal W_n:=\left\{v: \left|h(v)-\frac{R}{2\alpha}\right|\leq A_0\log\log n\right\}
\]
for a fixed $A_0>0$. Then, with high probability,
\begin{equation}
\sum_{v\notin\mathcal W_n}\de(v)^p=o(M_p), \qquad |\mathcal W_n|\leq(\log n)^{O(1)}.                       \label{eq:appendix-extreme-window}
\end{equation}

Also with high probability, uniformly for $v\in\mathcal W_n$,
\[
\de(v)=n^{1/(2\alpha)}(\log n)^{O(1)}.
\]
\end{lemma}

\begin{proof}
Write $t_*=R/(2\alpha)$ and
\[
L_n:=\sum_{v:h(v)<t_*-A_0\log\log n}\de(v)^p.
\]

The expected number of Poisson points above height $t_*+A_0\log\log n$ is $O((\log n)^{-\alpha A_0})$, so with high probability there are no such points. Since $p>2\alpha$, Palm--Mecke and the Poisson moment bound give
\begin{equation}
\e L_n =O\left(n^{p/(2\alpha)} (\log n)^{-A_0(p/2-\alpha)}\right).                \label{eq:extreme-window-tail}
\end{equation}

Put $b_n=n^{p/(2\alpha)}$. By Proposition~\ref{prop:giant-degree-moments}, $M_p/b_n$ is bounded away from zero in probability. To make the ratio argument explicit, fix $\varepsilon,\zeta>0$ and choose $c_\varepsilon>0$ such that
\[
\limsup_{n\to\infty}\p(M_p<c_\varepsilon b_n)<\varepsilon.
\]

Then
\[
\p(L_n>\zeta M_p) \leq \p(M_p<c_\varepsilon b_n) +\frac{\e L_n}{\zeta c_\varepsilon b_n} \leq \varepsilon+o(1).
\]

Since $\varepsilon$ and $\zeta$ are arbitrary, $L_n/M_p\to0$ in probability. Together with the absence of points above the upper edge of the window, this proves the first assertion in
\eqref{eq:appendix-extreme-window}.

The number of Poisson points in the deterministic window $|h-t_*|\leq A_0\log\log n$ has mean at most a fixed power of $\log n$. A Poisson tail bound therefore gives $|\mathcal W_n|\leq(\log n)^{O(1)}$ with high probability.

Finally, the lower edge of the window is $R/(2\alpha)-O(\log\log n)$. It exceeds both the degree-comparison height $2\log R+\omega(1)$ and the giant-core height $2\log R/(1-\alpha)$ for all large $n$. Thus,
\cite[Proposition~13, Remark~20, and Theorem~16]{dieter_mixing}, together
with a union bound over the polylogarithmic window, gives
\[
\de(v)\asymp e^{h(v)/2} =n^{1/(2\alpha)}(\log n)^{O(1)}
\]
uniformly for $v\in\mathcal W_n$.
\end{proof}

The next lemma identifies the degrees at polynomial ranks inside a fixed angular sector. The main point is to pass from counts ordered by height to vertices ordered by degree. The proof treats explicitly the vertices below the range of the degree comparison.

\begin{lemma}[Degree order within a sector]
\label{lem:sector-critical-orders}
Let $\mathcal I$ be a fixed angular arc of positive length, $0<\delta<1-\alpha$, and list the vertices of $\cG$ in $\mathcal I$ in decreasing degree order as $v_1,v_2,\ldots$. Uniformly for $(\log n)^2\leq j\leq n^\delta$, with high probability,
\begin{equation}
\de(v_j)\asymp e^{h(v_j)/2}, \qquad \de(v_j)^{2\alpha}\asymp\frac nj, \qquad h(v_j)=\frac1\alpha\log\frac nj+O(1).                   \label{eq:sector-critical-orders}
\end{equation}
\end{lemma}

\begin{proof}
Set
\[
t_j:=\frac1\alpha\log\frac nj,\qquad A_{\mathcal I}(t):= \left|\{v\in V(\cG):\vartheta(v)\in\mathcal I,\ h(v)\geq t\}\right|.
\]

We first establish a height-count estimate on one event, uniformly over the entire range of $j$. Let
\[
u_n=\left\lfloor\frac{R}{2\alpha} -\frac{\log R}{\alpha}-\log\log R\right\rfloor.
\]

Because $j\geq(\log n)^2$, for every fixed $L<\infty$ one has $t_j+L\leq u_n$ for all large $n$. Because $j\leq n^\delta$ and $\delta<1-\alpha$, there is a fixed $\varepsilon_0>0$ such that $t_j\geq(1/2+\varepsilon_0)R$ uniformly in $j$.

Apply the Poisson concentration estimate in
\cite[Lemma~12]{dieter_mixing} to the intersections of $\mathcal I$ with
the unit-height layers appearing in
\cite[Proposition~14(i)]{dieter_mixing}. All these layers are in the giant
core. A union bound over the $O(R)$ relevant layers gives fixed constants $0<c_A<C_A<\infty$ such that, with high probability, for every fixed $L<\infty$ and every admissible $j$,
\begin{equation}
c_A j e^{\mp\alpha L} \leq A_{\mathcal I}(t_j\pm L) \leq C_A j e^{\mp\alpha L}.                             \label{eq:sector-height-counts}
\end{equation}

Here points above $u_n$ cause no problem: their total number is $R^{1+o(1)}=o((\log n)^2)$ with high probability, so they can be absorbed into the constants uniformly over the stated range.

On the same event, \cite[Proposition~13]{dieter_mixing} gives fixed $0<c_d<C_d<\infty$ such that
\begin{equation}
c_de^{h(v)/2}\leq\de(v)\leq C_de^{h(v)/2}                \label{eq:sector-degree-height}
\end{equation}
for every vertex with height at least $\nu_0:=2\log R+\omega(1)$. This is uniform because Proposition~13 holds with extremely high probability and there are only polynomially many vertices. Vertices with height below $\nu_0$ satisfy
\begin{equation}
\de(v)\leq(\log n)^{1+o(1)}                              \label{eq:boundary-polylog-degree}
\end{equation}
uniformly, again by Proposition~13 and a union bound.

Choose fixed constants $D,L$ in this order. First take $D>0$ so large that
\[
C_de^{-D/2}<c_d.
\]
Then take $L>D$ so large that
\[
C_Ae^{-\alpha(L-D)}<1,\qquad c_Ae^{\alpha(L-D)}>1.
\]

Equation~\eqref{eq:sector-height-counts} then gives, uniformly in $j$,
\begin{equation}
A_{\mathcal I}(t_j+L-D)<j,\qquad A_{\mathcal I}(t_j-L+D)>j.                              \label{eq:sector-count-separation}
\end{equation}

The same estimate at $t_j-L$ gives at least $j$ vertices of height at least $t_j-L$. By \eqref{eq:sector-degree-height}, all of them have degree at least $c_de^{(t_j-L)/2}$. Hence,
\begin{equation}
\de(v_j)\geq c_de^{(t_j-L)/2} \geq n^{(1-\delta)/(2\alpha)+o(1)}.                     \label{eq:vj-polynomial-degree}
\end{equation}

In particular, \eqref{eq:boundary-polylog-degree} implies that $h(v_j)\geq\nu_0$ uniformly in $j$. Thus,
\eqref{eq:sector-degree-height} applies to $v_j$ itself. Put
$H:=h(v_j)$.

We now compare all other vertices with $v_j$. If a vertex $x$ satisfies $h(x)\leq H-D$ and $h(x)\geq\nu_0$, then
\[
\de(x)\leq C_de^{(H-D)/2} <c_de^{H/2}\leq\de(v_j).
\]

If instead $h(x)<\nu_0$, then
\eqref{eq:boundary-polylog-degree} and
\eqref{eq:vj-polynomial-degree} give the same strict inequality for all
large $n$. Therefore, every vertex of degree at least $\de(v_j)$ has height larger than $H-D$. Since there are at least $j$ vertices with degree at least $\de(v_j)$,
\begin{equation}
j\leq A_{\mathcal I}(H-D).                              \label{eq:rank-lower-count}
\end{equation}

Similarly, if $h(x)\geq H+D$, then
\[
\de(x)\geq c_de^{(H+D)/2} >C_de^{H/2}\geq\de(v_j).
\]
There can be at most $j-1$ vertices with degree strictly larger than $\de(v_j)$, and hence
\begin{equation}
A_{\mathcal I}(H+D)<j.                                  \label{eq:rank-upper-count}
\end{equation}

If $H>t_j+L$, then $A_{\mathcal I}(H-D)\leq A_{\mathcal I}(t_j+L-D)<j$ by
\eqref{eq:sector-count-separation}, contradicting
\eqref{eq:rank-lower-count}. If $H<t_j-L$, then
$A_{\mathcal I}(H+D)\geq A_{\mathcal I}(t_j-L+D)>j$, contradicting \eqref{eq:rank-upper-count}. Thus,
\[
h(v_j)=t_j+O(1)
\]
uniformly in $j$. Equation~\eqref{eq:sector-degree-height} now gives
\[
\de(v_j)\asymp e^{t_j/2}=(n/j)^{1/(2\alpha)}, \qquad \de(v_j)^{2\alpha}\asymp n/j,
\]
which proves \eqref{eq:sector-critical-orders}.
\end{proof}

The following uniform comparison is used when the main argument ranges over all vertices above a fixed fraction of the radial scale. Its weaker exponential form is sufficient there.

\begin{lemma}[Uniform comparison of degree and height]
\label{lem:uniform-degree-height}
For every fixed $\gamma>0$, with high probability, uniformly over all vertices with $h(v)\geq\gamma R$,
\begin{equation}
\de(v)=e^{h(v)/2+o(R)}.                                 \label{eq:uniform-degree-height}
\end{equation}
\end{lemma}

\begin{proof}
For all large $n$, $\gamma R$ exceeds both the degree-comparison height $2\log R+\omega(1)$ in \cite[Proposition~13]{dieter_mixing} and the giant-core height $2\log R/(1-\alpha)$ from
\cite[Remark~20 and Theorem~16]{dieter_mixing}. Proposition~13 holds with
extremely high probability, so a union bound over all $O(n)$ vertices gives fixed constants $0<c<C<\infty$ such that, simultaneously for every $h(v)\geq\gamma R$,
\[
ce^{h(v)/2}\leq\de(v)\leq Ce^{h(v)/2}.
\]

All such vertices belong to the giant, so their full and giant degrees coincide. Since fixed multiplicative constants are $e^{o(R)}$, the claimed form follows.
\end{proof}

The final lemma shows that vertices below a small fixed fraction of the height scale carry only a proportional part of the critical moment. This estimate allows the critical source law to discard the low-height region at a controlled cost.

\begin{lemma}[Critical moment below a fixed height]
\label{lem:truncated-critical-moment}
There is a constant $C<\infty$ such that, for every sufficiently small fixed $\gamma>0$, with high probability,
\begin{equation}
\frac1{M_{2\alpha}} \sum_{v:h(v)<\gamma R}\de(v)^{2\alpha} \leq C\gamma.                                           \label{eq:truncated-critical-moment}
\end{equation}
\end{lemma}

\begin{proof}
Put
\[
\ell_n=\left\lceil\max\left\{2\log R+\log\log R, \frac{2\log R}{1-\alpha}\right\}\right\rceil.
\]

By \eqref{eq:critical-palm-bound},
\[
\e\left[\sum_{v:h(v)<\ell_n}\de(v)^{2\alpha}\right] =O(n\log R)=o(nR).
\]

For every fixed $\gamma>0$, Markov's inequality therefore makes this contribution at most $\gamma nR$ with probability tending to one.

For $\ell_n\leq h(v)<\gamma R$, use the same unit-height layer event as in the proof of Proposition~\ref{prop:giant-degree-moments}. Each such layer contains $O(ne^{-\alpha m})$ vertices, and every vertex in the layer has degree $O(e^{m/2})$. Hence, each layer contributes $O(n)$ to the $2\alpha$-moment. There are at most $\gamma R+1$ such layers, so
\[
\sum_{v:\ell_n\leq h(v)<\gamma R}\de(v)^{2\alpha} =O(\gamma nR)
\]
with high probability.  The implicit constants here are independent of $\gamma$. Finally, \eqref{eq:critical-degree-moment} gives $M_{2\alpha}=\Omega(nR)$ with high probability. Combining the three estimates makes the ratio in \eqref{eq:truncated-critical-moment} $O(\gamma)$ with a $\gamma$-independent implicit constant, so one fixed choice of the constant in the statement proves the result.
\end{proof}

\end{document}